\documentclass[a4,12pto] {article}
\author{Mar\'{\i}a de la Paz Tirado Hern\'andez}
\usepackage{amssymb}
\usepackage{amsfonts}
\usepackage{latexsym}
\usepackage{amsmath}
\usepackage[latin1]{inputenc}
\usepackage[all]{xy}
\usepackage{theorem}
\usepackage{graphicx}
\usepackage{pb-diagram}
\usepackage{verbatim}
\usepackage{rawfonts}
\usepackage{anysize}
\usepackage{cancel}
\usepackage{color}
\usepackage[colorlinks=true,linkcolor=blue,urlcolor=black]{hyperref}
\usepackage{stmaryrd}
\usepackage{float}

\usepackage{tikz}
\usepackage{tikz-cd}

\usepackage[cal=boondox,calscaled=.96]{mathalfa}

\usepackage{fancyhdr}
\marginsize{2cm}{2cm}{2.5cm}{2.5cm}

\newtheorem{teo}{Theorem}[section]
\newtheorem{prop}[teo]{Proposition}
\newtheorem{cor}[teo]{Corollary}
\newtheorem{lem}[teo]{Lemma}
\newtheorem{Def}[teo]{Definition}
\newtheorem{nota}[teo]{Remark}

\newtheorem{ej}[teo]{Examples}

\newenvironment{proof}{
\noindent {\bf  Proof.}\ }{\hfill $\square$\vspace{1ex}\par}

\DeclareMathOperator{\supp}{\rm supp}
\DeclareMathOperator{\Sym}{\rm Sym}
\DeclareMathOperator{\C}{\rm \mathcal C}
\DeclareMathOperator{\s}{\rm \mathcal S}
\DeclareMathOperator{\T}{\rm \mathcal T}
\DeclareMathOperator{\gr}{\rm gr} 
\DeclareMathOperator{\End}{End}
\DeclareMathOperator{\U}{\mathcal{U}}
\DeclareMathOperator{\IDer}{\rm IDer}
\DeclareMathOperator{\Der}{\rm Der}
\DeclareMathOperator{\DD}{\mathcal{D}}
\DeclareMathOperator{\HS}{\rm HS}
\DeclareMathOperator{\Id}{\rm Id}
\DeclareMathOperator{\II}{\rm \mathbb I}

\newcommand{\N}{\mathbb{N}}
\newcommand{\Z}{\mathbb{Z}}
\newcommand{\Q}{\mathbb{Q}}

\newcommand{\pcirc}{{\scriptstyle \,\circ\,}}
\newcommand{\sbullet}{{\hspace{.1em}\scriptstyle\bullet\hspace{.1em}}}
\newcommand{\bfs}{{\bf s}}
\newcommand{\bft}{{\bf t}}

\author{Luis Narv\'{a}ez Macarro and Mar\'{\i}a de la Paz Tirado Hern\'andez\thanks{Boths authors are partially supported by MTM2016-75027-P, US-1262169
 and FEDER. Second author is partially suported by P12-FQM-2696. }}

\title{On the bracket of integrable derivations}

\date{January 2021}

\begin{document}

\maketitle

\begin{abstract}

We prove that any multi-variate Hasse--Schmidt derivation can be
decomposed in terms of substitution maps and uni-variate
Hasse--Schmidt derivations. As a consequence we prove that the
bracket of two $m$-integrable derivations is also $m$-integrable,
for $m$ a positive integer or infinity.

\bigskip

\noindent Keywords: Hasse--Schmidt derivation, integrable
derivation, differential operator, substitution map.

\noindent {\sc MSC: 14F10, 13N10, 13N15.}
\end{abstract}

\section*{Introduction}

Let $k$ be a commutative ring and $A$ a commutative $k$-algebra.
Given a positive integer $m$, or $m=\infty$,
 a $k$-linear derivation $\delta:A \to A$ is said to be $m$-integrable if
it extends up to a Hasse--Schmidt derivation
$D=(\Id,D_1=\delta,D_2,\dots)$ of $A$ over $k$ of length $m$. This
condition is automatically satisfied for any $m$ if $k$ contains the
rational numbers and $A$ is arbitrary, or if $k$ is arbitrary and
$A$ is a smooth $k$-algebra. The set $\IDer_k(A;m)$ of
$m$-integrable derivations of $A$ over $k$ is an $A$-module. A
natural question, suggested for instance by \cite[\S 3]{nar_2012}
and \cite{Ti2},
 is whether the (Lie) bracket $[\delta,\varepsilon]= \delta\varepsilon-\varepsilon\delta$ of two $m$-integrable derivations $\delta, \varepsilon$ is $m$-integrable or not, in the case of course where $\IDer_k(A;m) \subsetneq \Der_k(A)$. 
 The fact that the modules $\IDer_k(A;m)$ are closed under Lie brackets seems like a very basic property, necessary for any reasonable behavior that we can expect of these objects as differential invariants of singularities in nonzero characteristics, and as far as we know it has not been proven in the existing literature.
\medskip

 If we take two $m$-integrals of our derivations
$$ D=(\Id,D_1=\delta,D_2,\dots),\quad E=(\Id,E_1=\varepsilon,E_2,\dots),
$$
their commutator (in the group of Hasse--Schmidt derivations of
length $m$) has the form
$$
D \pcirc E\pcirc D^* \pcirc E^* =
(\Id,0,[D_1,E_1]=[\delta,\varepsilon],\dots ),
$$
where $D^*$ denotes the inverse of $D$ for the group structure of Hasse--Schmidt derivations,
but it is not clear how to produce a Hasse--Schmidt derivation of
length $m$ such that its $1$-component is $[\delta,\varepsilon]$, if
it exists.
\medskip

In this paper we show how multi-variate Hasse--Schmidt derivations
allow us to answer the above question. Let us see what happens in
the simple case of length $m=2$. Consider the external product
$F=D\boxtimes E =( F_{(i,j)})_{0\leq i,j\leq 2}$, with $ F_{(i,j)} =
D_i \pcirc E_j$, which is a 2-variate Hasse--Schmidt derivation, and
the composition
$$
G= (D\boxtimes E) \pcirc (D^*\boxtimes E^*).
$$
First, one checks that $G_{(1,0)} = G_{(2,0)} = G_{(0,1)} =
G_{(0,2)} = 0$, and from there we deduce easily that the
``restriction of $G$ to the diagonal'', i.e.
$G'=\left(G_{(0,0)}=\Id, G_{(1,1)}, G_{(2,2)}\right)$, is a
(uni-variate) Hasse--Schmidt derivation of length $2$. But
$G_{(1,1)}$ turns out to be $[D_1,E_1]=[\delta,\varepsilon]$, and so
$[\delta,\varepsilon]$ is $2$-integrable. Actually, the explicit
expression of $G_{(2,2)}$ is
\begin{eqnarray*}
& D_2 \pcirc E_2 - D_2 \pcirc E_1^2  - D_1 \pcirc E_2 \pcirc D_1 +
D_1 \pcirc E_1 \pcirc D_1 \pcirc E_1 + E_2 \pcirc D_1^2 -  E_2
\pcirc D_2 - E_1 \pcirc D_1^2  \pcirc E_1 + E_1 \pcirc D_2 \pcirc
E_1.
\end{eqnarray*}

In order to generalize the above idea to arbitrary length, we need a
decomposition result which allows us to express any
$\Delta$-variate Hasse--Schmidt derivation $D$, for $p\geq 1$
and $\Delta\subset \N^p$ a finite co-ideal, as the ordered
composition (remember that the group of $\Delta$-variate
Hasse--Schmidt derivations under composition is not abelian in
general)  of a totally ordered finite family of $\Delta$-variate
Hasse--Schmidt derivations, each one obtained as the action of a
monomial substitution map on a uni-variate Hasse--Schmidt
derivation. When $\Delta$ is infinite, a similar result holds, but
our totally ordered familly becomes infinite. Moreover, the above
decomposition is unique if we fix the substitution maps we are
using, and it is governed by the arithmetic combinatorics of $\N^p$ (see Theorem \ref{Descomposicion} for more details). We think that such a decomposition is interesting in itself: it can be understood as a structure theorem of multi-variate Hasse--Schmidt derivations.
\medskip

Let us comment on the content of the paper.

In Section 1, we recall the basic notions, constructions and
notations about Hasse--Schmidt derivations, substitution maps and
integrability.
\medskip

In Section 2,  we describe an arithmetic partition of $\N^p
\setminus \{0\}$, we define a total ordering on it and we study its
behavior with respect to the addition in $\N^p$.
\medskip

Section 3 contains the main results of this paper, namely the
decomposition theorem of multi-variate Hasse--Schmidt derivations in
terms of uni-variate Hasse--Schmdit derivations and substitution
maps (see Theorem \ref{Descomposicion}), and the answer of the
motivating question of this paper: the bracket of $m$-integrable
derivations is $m$-integrable too (see Corollary \ref{El conmutador
es integrable}).
\medskip

In Section 4, we apply the previous results to exhibit a natural
Poisson structure on the divided power algebra of the module of
integrable derivations, and we prove its compatibility with the
canonical Poisson structure of the graded ring of the ring of
differential operators by means of the map $\vartheta_{A/k}$ of
\cite[Section (2.2)]{nar_2009}.

\section{Preliminaries and Notations}

Throughout this paper, $k$ will be a commutative ring and $A$ a
commutative $k$-algebra, and in this section $M$ will be an abelian
group and $R$ a ring, not-necessarily commutative.
\medskip

Let $p\geq 1$ be an integer. The monoid $(\N^p,+)$ is endowed with a
natural partial ordering: for $\alpha,\beta\in \N^p$,
$$ \alpha \leq \beta\quad \stackrel{\text{def.}}{\Longleftrightarrow}\quad \exists \gamma \in \N^p\ \text{such that}\ \beta = \alpha +
\gamma\quad \Longleftrightarrow\quad \forall i=1\dots,p, \quad
\alpha_i \leq \beta_i.
$$
Let $\N^p_+:=\N^p \setminus \{(0,\ldots,0)\}$ and let
$|\alpha|:=\alpha_1+\cdots+\alpha_p$ for any $\alpha\in \N^p$.

\medskip

Let $\bfs = \{s_1,\dots,s_p\}$ be a set of $p$ many variables. The
abelian group $M\llbracket\bfs\rrbracket$ will be always considered
as a topological $\Z\llbracket\bfs\rrbracket$-module with the
$\langle \bfs \rangle$-adic topology.

\begin{Def} \label{def:ideal-co-ideal}
We say that a subset $\Delta \subset \N^p$ is  a {\em co-ideal} of
$\N^p$ if $\alpha'\in\Delta$ whenever $\alpha'\leq \alpha$ and
$\alpha\in\Delta$.
\end{Def}

For each co-ideal $\Delta \subset \N^p$, we denote by $\Delta_M$ the
closed sub-group of $M\llbracket\bfs\rrbracket$ whose elements are
the formal power series $\sum_{\alpha\in\N^p} m_\alpha \bfs^\alpha$
such that $m_\alpha=0$ whenever $\alpha\in \Delta$, and
$M\llbracket\bfs\rrbracket_\Delta :=
M\llbracket\bfs\rrbracket/\Delta_M$. Any element $m\in
M\llbracket\bfs\rrbracket_\Delta$ can be written in a unique way
$m=\sum_{\alpha\in \Delta} m_\alpha \bfs^\alpha$, and its support is
$\supp(m) = \{ \alpha \in \Delta\ |\  m_\alpha \neq 0\} \subset
\Delta$. Let us notice that $M\llbracket\bfs\rrbracket_{\N^p} =
M\llbracket\bfs\rrbracket$ (the case of $\Delta = \N^p$).
\medskip

If $M$ is a ring, say $M=R$,  then $\Delta_R$ is a closed two-sided
ideal of $R\llbracket\bfs\rrbracket$ and so
$R\llbracket\bfs\rrbracket_\Delta$ is a topological ring, which we
always consider endowed with the $\langle \bfs \rangle$-adic
topology ($=$ to the quotient topology).
\medskip

For non-empty co-ideals $\Delta' \subset \Delta$  of $\N^p$, we have
natural $\Z\llbracket\bfs\rrbracket$-linear projections
$\tau_{\Delta \Delta'}:M\llbracket\bfs\rrbracket_{\Delta}
\longrightarrow M\llbracket\bfs\rrbracket_{\Delta'}$, that we call
{\em truncations}:
$$\tau_{\Delta \Delta'} : \sum_{\scriptscriptstyle \alpha\in\Delta} m_\alpha \bfs^\alpha \in M\llbracket\bfs\rrbracket_{\Delta} \longmapsto \sum_{\scriptscriptstyle \alpha\in\Delta'} m_\alpha \bfs^\alpha \in
M\llbracket\bfs\rrbracket_{\Delta'}.
$$
If $M=R$ is a ring, then the truncations $\tau_{\Delta \Delta'}$ are
ring homomorphisms.
\medskip

We denote by $\U(R;\Delta)$ the multiplicative sub-group of the
units of $R\llbracket\bfs\rrbracket_\Delta$ whose 0-degree
coefficient is $1$. When $p=1$ and $\Delta=\{0,\dots,m\}$, we simply
denote $\U(R;m):=\U(R;\{0,\dots,m\})$. The multiplicative inverse
of a unit $r\in R\llbracket\bfs\rrbracket_\Delta$ will be denoted by
$r^*$. For $\Delta\subset \Delta'$ co-ideals we have
$\tau_{\Delta'\Delta}\left(\U(R;\Delta')\right) \subset
\U(R;\Delta)$ and the truncation map
$\tau_{\Delta'\Delta}:\U(R;\Delta') \to \U(R;\Delta)$ is a group
homomorphism. Clearly, we have:
\begin{equation} \label{eq:U-inv-limit-finite}
   \U(R;\Delta)
   = \lim_{\stackrel{\longleftarrow}{\substack{\scriptscriptstyle \Delta' \subset \Delta\\ \scriptscriptstyle  \sharp \Delta'<\infty}}} \U(R;\Delta').
\end{equation}

\medskip

\begin{Def}  \label{def:infinite-comp} Let $(I,\preceq)$ be a totally ordered set, possibly infinite, and $\mathbf{r}=(r_i)_{i\in I}$ a family of elements in $\U(R;\Delta)$. We say that this family is {\em composable} if for each finite co-ideal $\Delta' \subset \Delta$, the set $I_{\Delta'}=\{i\in I\ |\ \tau_{\Delta\Delta'}(r_i) \neq 1\}$ is finite. In such a case, for each finite co-ideal $\Delta' \subset \Delta$ we define
$$
C_{\Delta'}(\mathbf{r}) := \tau_{\Delta\Delta'}(r_{i_1}) \pcirc
\cdots \pcirc \tau_{\Delta\Delta'}(r_{i_m})  \in \U(R;\Delta'),
$$
where $I_{\Delta'} = \{i_1,\ldots,i_m\}$ and $i_1 \prec \cdots \prec
i_m$. It is clear that if $\Delta'' \subset \Delta' $ is another
finite co-ideal, we have $I_{\Delta''} \subset I_{\Delta'}$ and
$\tau_{\Delta'\Delta''}(C_{\Delta'}(\mathbf{r})) =
C_{\Delta''}(\mathbf{r})$, and so we define the {\em ordered
composition} of the family $\mathbf{r}$ as (see
(\ref{eq:U-inv-limit-finite}))
$$
\circ_{i\in I}\, r_i =
\lim_{\stackrel{\longleftarrow}{\substack{\scriptscriptstyle \Delta'
\subset \Delta\\ \scriptscriptstyle  \sharp \Delta'<\infty}}}
C_{\Delta'}(\mathbf{r}) \in \U(R;\Delta).
$$
\end{Def}

Let $p,q\geq 1$ be integers,
$\bfs=\{s_1,\dots,s_p\},\bft=\{t_1,\dots,t_q\}$ two sets of
variables and $\Delta\subset \N^p, \nabla\subset \N^q$ non-empty
co-ideals.

\begin{Def}  \label{def:substitution-maps}
An $A$-algebra map $\varphi:A\llbracket\bfs\rrbracket_\Delta
\xrightarrow{} A\llbracket\bft\rrbracket_\nabla$ will be called a
{\em substitution map} whenever $\varphi(s_i) \in \langle \bft
\rangle$ for all $i=1,\dots, p$. Such a map is continuous and
uniquely determined by the images $\varphi(s_i), i=1,\dots, p$. A
substitution map $\varphi:A\llbracket\bfs\rrbracket_\Delta
\xrightarrow{} A\llbracket\bft\rrbracket_\nabla$ will be called {\em
monomial} if $\varphi(s_i)$ is a monomial in $\bft$ for all
$i=1,\dots,p$.
\end{Def}

\begin{Def} \label{def:HS}
A {\em $\Delta$-variate Hasse--Schmidt derivation}, or a {\em
$\Delta$-va\-riate HS-de\-ri\-va\-tion} for short, of $A$ over
$k$
 is a family $D=(D_\alpha)_{\alpha\in \Delta}$
 of $k$-linear maps $D_\alpha:A
\longrightarrow A$, satisfying the following Leibniz type
identities: $$ D_0=\Id_A, \quad
D_\alpha(xy)=\sum_{\beta+\gamma=\alpha}D_\beta(x)D_\gamma(y) $$ for
all $x,y \in A$ and for all $\alpha\in \Delta$. We denote by
$\HS^p_k(A;\Delta)$ the set of all $\Delta$-variate
HS-derivations of $A$ over $k$. For $p=1$, a uni-variate
HS-derivation will be simply called a {\em Hasse--Schmidt
derivation}   (a HS-derivation for short), or a {\em higher
derivation}\footnote{This terminology is used for instance in
\cite[\S 27]{mat_86}.}, and we will simply write $\HS_k(A;m):=
\HS^1_k(A;\{0,\ldots,m\})$.\footnote{These HS-derivations are called
of length $m$ in \cite[\S 27]{mat_86}.}
\end{Def}

Any $\Delta$-variate HS-derivation $D$ of $A$ over $k$  can be
understood as a power series $ \sum_{\scriptscriptstyle
\alpha\in\Delta} D_\alpha \bfs^\alpha \in
R\llbracket\bfs\rrbracket_\Delta$, with $R=\End_k(A)$, and so we
consider $\HS^p_k(A;\Delta) \subset
R\llbracket\bfs\rrbracket_\Delta$. Actually, $\HS^p_k(A;\Delta)$ is
a (multiplicative) sub-group of $\U(R;\Delta)$. The group
operation in $\HS^p_k(A;\Delta)$ is explicitly given by $(D \pcirc
E)_\alpha = \sum_{\scriptscriptstyle \beta+\gamma=\alpha} D_\beta
\pcirc E_\gamma$, and the identity element of $\HS^p_k(A;\Delta)$ is
$\mathbb{I}$ with $\mathbb{I}_0 = \Id$ and $\mathbb{I}_\alpha = 0$
for all $\alpha \neq 0$. The inverse of a $D\in \HS^p_k(A;\Delta)$,
in the sense of the group structure on $\mathcal \U(A; \Delta)$,
will be denoted by $D^*$.
\medskip

For $\Delta' \subset \Delta \subset \N^p$ non-empty co-ideals, we
have truncations $\tau_{\Delta \Delta'}: \HS^p_k(A;\Delta)
\longrightarrow \HS^p_k(A;\Delta'), $ which are group homomorphisms.
\medskip

For each substitution map $\varphi:A\llbracket\bfs\rrbracket_\Delta
\to A\llbracket\bft\rrbracket_\nabla$ and each HS-derivation $D=
\sum_{\scriptscriptstyle \alpha\in\Delta} D_\alpha \bfs^\alpha \in
\HS^p_k(A;\Delta)$, we know that $\varphi \sbullet D =
\sum_{\scriptscriptstyle \alpha\in\Delta} \varphi(\bfs^\alpha)
D_\alpha $ is a $\nabla$-variate HS-derivation (see
\cite[Proposition 10]{Na3}).
\medskip

\begin{Def} (Cf. \cite{brown_1978,mat-intder-I,nar_2012})  \label{def:HS-integ}
Let $m\geq 1$ be an integer or $m=\infty$, and $\delta:A\to A$ a
$k$-derivation.
 We say that $\delta$ is {\em $m$-integrable} (over $k$)  if there is a
HS-derivation $D\in \HS_k(A;m)$ such that $D_1=\delta$. A such $D$
is called an $m$-integral of $\delta$. The set of $m$-integrable
$k$-derivations of $A$ is denoted by $\IDer_k(A;m)$. We say that
$\delta$ is {\em $f$-integrable} {\em (finite integrable)} if it is
$m$-integrable for all \underline{integers} $m\geq 1$. The set of
$f$-integrable $k$-derivations of $A$ is denoted by $\IDer^f_k(A)$.
\end{Def}

The sets $\IDer_k(A;m)$ and $\IDer^f_k(A)$ are $A$-submodules of
$\Der_k(A)$, and we have
$$ \Der_k(A) = \IDer_k(A;1) \supset \IDer_k(A;2) \supset \cdots \supset \IDer^f_k(A) \supset \IDer_k(A;\infty).
$$
If $\Q\subset k$ or $A$ is $0$-smooth over $k$, then any
$k$-derivation of $A$ is $\infty$-integrable, and so $\Der_k(A) =
\IDer^f_k(A) = \IDer_k(A;\infty)$ (see \cite[p. 230]{mat-intder-I}).
\medskip

The following Proposition is a straightforward consequence of Theorems 3.14 and 4.1 of \cite{Ti1} and will be used in section 3. 

\begin{prop}\label{Prop 3.14}
Let $k$ be a ring of characteristic $p >0$, $e, s \geq 1$  two integers and $D\in \HS_k(A; ep^s)$ a HS-derivation with $D_1=D_2=\ldots=D_{e-1}=0$. Then, $D_e\in \IDer_k(A;p^s)$.  
\end{prop}

\section{An ordered partition}

In this section, we define an ordered partition of $\N^q_+$ of arithmetic nature that will be crucial for the proof of our main results in Section 3.
\medskip

Let $q\geq 2$ be an integer. For $\beta_1,\ldots,\beta_q\in \Z$, we
denote by $\gcd(\beta_1,\ldots,\beta_q)$ the (unique) non-negative
integer $g$ such that $\Z\beta_1 + \cdots + \Z\beta_q = \Z g$.
Notice that $\gcd(\beta_1,\ldots,\beta_q)=0$ if and only if the
ideal  $(\beta_1,\ldots,\beta_q)$ is equal to $0$.

\begin{Def}   For $\alpha,\beta\in \N^q_+$, we define
$$ \alpha \sim \beta \quad \stackrel{\text{\rm def.}}{\Longleftrightarrow} \quad \exists r\in \Q^\times\ |\  \beta = r \alpha.
$$
\end{Def}
It is clear that $\sim$ is an  equivalence relation in $\N^q_+$.

\begin{Def}\label{Conjunto C}
We define $\mathcal C^q$ as the set $\{(\beta_1, \ldots, \beta_q)\in
\mathbb N^q_+ | \gcd(\beta_1,\ldots, \beta_q) = 1\}$.
\end{Def}

\begin{lem} \label{lem:parametrizacion-particion}
With the above notations, the map $ \beta \in \C^q \longmapsto
[\beta]  \in \N^q_+/\sim $ is bijective. Moreover, for each $\beta
\in \C^q$, the equivalence class $[\beta]$ coincide with the set
$\N_+ \beta =\{r \beta\ |\ r\in \N_+\}$.
\end{lem}

\begin{Def}\label{Aplicacion g}
We define the map $g^q:\N^q \to \C^2\cup \{(0,0)\}$ (or simply $g$
if there is no confusion) as:
$$
\begin{array}{rcclcc}
\beta& \mapsto & \left\{\begin{array}{ll} (0,0)& \mbox{ if }
\beta_1=\beta_2=0\\[0.2cm]
\dfrac{1}{\gcd(\beta_1,\beta_2)}(\beta_1,\beta_2)& \mbox{ otherwise.
}
 \end{array} \right.
\end{array}
$$
\end{Def}

Observe that, if $\beta\in\C^q$, then
$(\gcd(\beta_1,\beta_2),\beta_3,\ldots,\beta_q)\in \C^{q-1}$ and if
$\beta'\in[\beta]$, then $g(\beta')=g(\beta)$.
\medskip

We are going to define a total ordering $\preceq^q$ on $\C^q$, and
so on the partition $\N^q_+/\sim$ through the bijection from Lemma
\ref{lem:parametrizacion-particion}.
\medskip

Let us consider $\beta,\gamma\in \C^q$. If $q=2$, then
$\beta\prec^2 \gamma $ if and only if $\gamma_2 \beta_1 <
\gamma_1 \beta_2 $. For $q\geq 3$, we say that
$\beta\prec^q\gamma$ if some of the following conditions hold:
 \begin{itemize}
 \item[1.] $g(\beta)=(0,0)$ and $g(\gamma)\neq (0,0)$.
 \item[2.] $g(\beta),g(\gamma)\neq (0,0)$ and $g(\beta)\prec^2
 g(\gamma)$.
 \item[3.] $g(\beta)=g(\gamma)$ and $(\gcd (\beta_1,\beta_2),\beta_3,\ldots,
 \beta_q)\prec^{q-1} (\gcd (\gamma_1,\gamma_2),\gamma_3,\ldots, \gamma_q)$.
 \end{itemize}
 As usual, we say that $\beta\preceq^q\gamma$ if and only if $\beta\prec^q\gamma$ or $\beta=\gamma$.
\medskip

The proof of the following proposition can be easily proved by
induction on $q$ and it is left to the reader.

\begin{prop} The relation $\preceq^q$ above is a
total ordering on $\C^q$. Moreover,
$$\textstyle (0,\dots,0,1) = \min_{\preceq^q} (\C^q)\quad \text{and}\quad (1,0,\dots,0) = \max_{\preceq^q} (\C^q).
$$
\end{prop}

We will also denote by $\preceq^q$ the total ordering induced on
$\N^q_+/\sim$ by the bijection from Lemma
\ref{lem:parametrizacion-particion}.
\medskip

The following proposition deals with the behavior of the total
ordering $\preceq^q$ with respect to the monoide structure on
$\N^q$. It will be the main tool in proving the results in Section 3.

\begin{prop}\label{Desigualdad 1}
Let $\lambda,\sigma,\beta\in \N^q_+$ such that
$\lambda + \sigma = \beta$. Then, one and only one of the following properties holds:\\
(a) $[\lambda]=[\beta]=[\sigma]$,\\
(b) $[\sigma] \prec^q [\beta] \prec^q [\lambda]$,\\
(c) $[\lambda] \prec^q [\beta] \prec^q [\sigma]$.
\end{prop}

\begin{proof} It is clear that if any two $q$-tuples among
$\sigma$, $\beta$, $\lambda$ have the same class, then all three
classes $[\sigma]$, $[\beta]$, $[\lambda]$ are equal. So, let us
assume that they are all different, in particular $[\lambda]\neq
[\beta]$. Hence, we have either $[\lambda]\prec^q [\beta]$ or
$[\beta]\prec^q[\lambda]$. We will prove the result by induction on
$q\geq 2$.

If $q=2$ and $[\lambda]\prec^2[\beta]$, then $\beta_2\lambda_1 <
\beta_1\lambda_2$ and, since $\lambda+\sigma=\beta$, we obtain that
$\beta_2(\beta_1-\sigma_1)<\beta_1(\beta_2-\sigma_2)$. Hence,
$\beta_1\sigma_2 <\beta_2\sigma_1$ and, by definition,
$[\beta]\prec^2[\sigma]$. If $[\beta]\prec^2[\lambda]$, for similar
reasons as before, we deduce that $[\sigma]\prec^2[\beta]$ and
the proposition is proved for $q=2$. Let us assume that the result
is true for $q-1$ and we will prove it for $q \geq 3$. We will start
assuming that $[\lambda] \prec^q [\beta]$. Three different cases
have to be considered according to the definition of $\prec^q$:
\medskip

\noindent {\bf Case 1:} 
{ $g(\lambda)=(0,0)$ and $g(\beta)\neq (0,0)$.}
In this case, $\lambda_i=0$ for $i=1,2$ and so
$(\sigma_1,\sigma_2)=(\beta_1,\beta_2)$, which implies that
$g(\sigma)=g(\beta)$ and
$\gcd(\beta_1,\beta_2)=\gcd(\sigma_1,\sigma_2)=d\neq0$. From this,
and from the equality $\lambda+\sigma=\beta$, it follows that
$$
(0,\lambda_3,\ldots,\lambda_q)+(d,\sigma_3,\ldots,\sigma_q)=
(d,\beta_3,\ldots,\beta_q).
$$
Moreover, we have $g((0,\lambda_3,\ldots,\lambda_q))=(0,a)$ where
$a=0$ if $\lambda_3=0$ and $a=1$ otherwise. So, since the first
component of $g(d,\beta_3,\ldots,\beta_q)$ is not zero and
$(0,1)=\min_{\preceq^2} \C^2$, we deduce that
$[(0,\lambda_3,\ldots,\lambda_q)]\prec^{q-1}
[(d,\beta_3,\ldots,\beta_q)]$. By induction hypothesis, we have
$[(d,\beta_3,\ldots,\beta_q)] \prec^{q-1}
[(d,\sigma_3,\ldots,\sigma_q)]$ and we conclude that $[\beta]
\prec^q [\sigma]$.
\smallskip

\noindent {\bf Case 2:} { $g(\lambda),g(\beta)\neq (0,0)$ and $g(\lambda)\prec^2
g(\beta)$.} It is clear that $g(\gamma)=g((\gamma_1,\gamma_2))$ for
all $\gamma \in \mathbb N^q$ and, if $g(\gamma)\neq (0,0)$, then
$[g(\gamma)]=[(\gamma_1,\gamma_2)]$ for all $\gamma\in \mathbb
N^q_+$ because
$(\gamma_1,\gamma_2)=\gcd(\gamma_1,\gamma_2)g(\gamma)$. Therefore,
we have $[(\lambda_1,\lambda_2)]\prec^2[(\beta_1,\beta_2)]$, and
since $\lambda+\sigma=\beta$, we deduce that
$(\sigma_1,\sigma_2)\neq (0,0)$. Now, we apply induction
hypothesis to the equality
$(\lambda_1,\lambda_2)+(\sigma_1,\sigma_2)=(\beta_1,\beta_2)$ and we
get $[(\beta_1,\beta_2)]\prec^2[(\sigma_1,\sigma_2)]$, which implies
that $g(\beta)\prec^2 g(\sigma)$. So, by definition,
$[\beta]\prec^q[\sigma]$.
\smallskip

\noindent {\bf Case 3:} { $g(\lambda)=g(\beta)$ and
$[(\gcd(\lambda_1,\lambda_2),\lambda_3,\ldots, \lambda_q)]
\prec^{q-1} [(\gcd(\beta_1,\beta_2),\beta_3,\ldots, \beta_q)]$.} If
$g(\lambda)=g(\beta)=(0,0)$, then $g(\sigma)=(0,0)$ because
$\lambda_i=\beta_i=0$ for $i=1,2$ and $\lambda+\sigma=\beta$. If
$g(\lambda)=g(\beta)\neq (0,0)$, then
$[(\lambda_1,\lambda_2)]=[g(\lambda)]=[g(\beta)]=[(\beta_1,\beta_2)]$.
Let us notice that $(\sigma_1,\sigma_2)\neq (0,0)$ otherwise,
$g(\sigma)=(0,0)$ and $g(\beta)\neq (0,0)$ and from Case 1, we get
that $[\beta]\prec^q [\lambda]$, which is a contradiction. Now,
induction hypothesis can be applied to
$(\lambda_1,\lambda_2)+(\sigma_1,\sigma_2)=(\beta_1,\beta_2)$ and we
obtain $[(\sigma_1,\sigma_2)]=[(\beta_1,\beta_2)]$. So, in any case,
$g(\lambda)=g(\beta)=g(\sigma)=\tau$. Since
$(\gamma_1,\gamma_2)=\gcd(\gamma_1,\gamma_2)g(\gamma)$ for all
$\gamma\in \N^q$, we have that
$$
\gcd(\lambda_1,\lambda_2)g(\lambda)+\gcd(\sigma_1,\sigma_2)g(\sigma)=
\gcd(\beta_1,\beta_2)g(\beta).
$$
If $\tau=(0,0)$, then
$\gcd(\lambda_1,\lambda_2)=\gcd(\sigma_1,\sigma_2)=
\gcd(\beta_1,\beta_2)=0$, otherwise
$\gcd(\lambda_1,\lambda_2)+\gcd(\sigma_1,\sigma_2)=
\gcd(\beta_1,\beta_2)$. So, in both cases,
$$
(\gcd(\lambda_1,\lambda_2),\lambda_3,\ldots,\lambda_q)+
(\gcd(\sigma_1,\sigma_2),\sigma_3,\ldots,\sigma_q)=
(\gcd(\beta_1,\beta_2),\beta_3,\ldots, \beta_q).
$$
From $[(\gcd(\lambda_1,\lambda_2),\lambda_3,\ldots, \lambda_q)]\prec^q
[(\gcd(\beta_1,\beta_2),\beta_3,\ldots,\beta_q)]$ and the induction
hypothesis, we get that 
$$[(\gcd(\beta_1,\beta_2),\beta_3,\ldots,
\beta_q)]\prec^{q-1}
[(\gcd(\sigma_1,\sigma_2),\sigma_3,\ldots,\sigma_q)]$$ 
and, by
definition, $[\beta]\prec^q [\sigma]$.
\medskip

In conclusion, we have proven that $[\lambda] \prec^q[\beta]$ implies
$[\beta] \prec^q [\sigma]$. Now, let us assume that $[\beta]\prec^q
[\lambda]$. If $g(\beta)=(0,0)$ and $g(\lambda)\neq (0,0)$, we have
$\beta_i=0$ for $i=1,2$ and since $\sigma_i+\lambda_i=\beta_i$, we
deduce that $\sigma_i=\lambda_i=0$ for $i=1,2$, but
$(\lambda_1,\lambda_2)\neq (0,0)$, so we have a contradiction. The
cases when $g(\beta),g(\lambda)\neq (0,0)$ with $g(\beta)\prec^2
g(\lambda)$ and when $g(\beta)=g(\lambda)$ with
$[(\gcd(\beta_1,\beta_2),\beta_3,\ldots,\beta_q)]\prec^{q-1}
[(\gcd(\lambda_1,\lambda_2),\lambda_3,\ldots,\lambda_q)]$ are
similar to the previous Cases 2 and 3 respectively. Hence, we have
the result.
\end{proof}

\begin{lem}\label{La suma de vectores} Let $\lambda_1,\ldots, \lambda_s\in \N^q_+$ such that
$[\lambda_1] \prec^q [\lambda_2] \prec^q \cdots \prec^q
[\lambda_s]$. Then, $[\lambda_1]\prec^q
[\lambda_1+\cdots+\lambda_s]$.
\end{lem}

\begin{proof}
We will prove the lemma by induction on $s\geq 2$. By Proposition
\ref{Desigualdad 1}, since $[\lambda_1]\prec^q [\lambda_2]$, we get
$[\lambda_1] \prec^q [\lambda_1+\lambda_2]\prec^q [\lambda_2]$. Let
us assume that the result is true for $i<s$,  we will prove it for
$s>2$. By induction hypothesis, $[\lambda_2]\prec^q
[\lambda_2+\cdots+ \lambda_s]$. Since $[\lambda_1]\prec^q
[\lambda_2]$, by Proposition \ref{Desigualdad 1}, $[\lambda_1]
\prec^q [\lambda_1+\lambda_2+\cdots +\lambda_s]$ and we have the
result.
\end{proof}

\section{Main results}

From now on, $\Delta\subseteq \N^q$ will be a non-zero and non-empty
co-ideal and we will simply use $\prec$ and  $\preceq$ instead of $\prec^q$ and $\preceq^q$ (the above
total ordering on $\C^q$ or $\N^q_+/\sim$) if no confusion arises.

\medskip

We denote $ \C^q_\Delta=\C^q \cap \Delta$, and for each $\beta\in
\C^q_\Delta$, we define $P^\Delta_\beta := [\beta] \cap
\Delta=\{n\beta\in \N^q_+ \ | \ n\in \mathbb N_+, n\beta\in
\Delta\}$, $M_\beta^\Delta := \{n\in \N_+ \ |\ n\beta \in\Delta\} $,
and $m_\beta^\Delta=\#(M_\beta^\Delta)=\#(P^\Delta_\beta)$. Let us
notice that $m_\beta^\Delta=\max M_\beta^\Delta$ if $M_\beta^\Delta$
is finite and $m_\beta^\Delta=\infty$ otherwise. The
$P^\Delta_\beta$'s, $\beta \in \C^q_\Delta$, form the partition of
$\Delta\setminus \{0\}$ induced by $\sim$. For each $\beta \in
\C^q_\Delta$, we also introduce

\begin{eqnarray*}
&\displaystyle \T_\beta^\Delta:=\bigsqcup_{\beta \preceq \lambda}
P_\lambda^\Delta=\bigsqcup_{\beta \preceq \lambda} \{n\lambda\in
\N^q_+ \ |\  n\in \mathbb N_+, n\lambda \in \Delta\},&
\\
&\displaystyle \s_\beta^\Delta:= \Delta \setminus ( \T_\beta^\Delta
\cup \{0\} )=\bigsqcup_{\lambda \prec \beta }
P_\lambda^\Delta=\bigsqcup_{\lambda \prec \beta} \{n\lambda\in
\N^q_+ \ |\ n\in \mathbb N_+, n\lambda \in \Delta\},&
\\
\end{eqnarray*}
and the monomial substitution map
$$
\begin{array}{rccc}
\psi_{\beta,\Delta}:& A\llbracket\mu\rrbracket_{m_\beta^\Delta}&\to
&
A\llbracket s_1,\ldots,s_q\rrbracket_\Delta\\
&\mu&\mapsto & s_1^{\beta_1} \ldots s_q^{\beta_q}.
\end{array}
$$
where, for $m\in \N_+$, we define
$A\llbracket\mu\rrbracket_m=A\llbracket\mu\rrbracket_{\{ n\in \N \ |
\ n \leq m\}}$.
 \smallskip

It is clear that for any $D\in \HS^q_k(A;\Delta)$ with
$\supp(D)\subset \{0\} \cup P^\Delta_\beta$, the sequence $E:=(E_r:=
D_{r \beta})_{r\in M_\beta^\Delta \cup \{0\}}$ is a (uni-variate)
HS-derivation of length $m_\beta^\Delta$, and $D=
\psi_{\beta,\Delta} \sbullet E$. The following proposition generalizes this result and will be the main step in proving Theorem \ref{Descomposicion}.

\begin{prop}\label{Primer paso descomposicion}
 Let $\beta \in
\C^q_\Delta$, $m=m_\beta^\Delta$ and $D\in \HS^q_k(A;\Delta)$ such
that $\supp(D) \subset \T_\beta^\Delta \cup \{0\}$ (or equivalently,
$D_\gamma=0$ for all $\gamma\in \s_\beta^\Delta$). Then, there are
unique $E\in \HS_k(A;m)$ and $D'\in \HS^q_k(A;\Delta)$ such that
$\supp(D') \subset \T_\beta^\Delta \setminus P_\beta^\Delta \cup
\{0\}$ (or equivalently, $D'_\gamma = 0$ for all $\gamma\in
\s_\beta^\Delta \sqcup P^\Delta_\beta$) and $D= (\psi_{\beta,\Delta}
\sbullet E) \pcirc D'$. Moreover, if $D_\gamma=0$ for all $\gamma
\in P^\Delta_\beta$ with $\gamma \leq \alpha$ for some $\alpha \in
\Delta$, then  $D'_\gamma=D_\gamma$ for all $\gamma\in \Delta$ with
$\gamma \leq \alpha$.
\end{prop}

\begin{proof}
We start proving that the sequence $E:=(E_r:=D_{r\beta})\in
\HS_k(A;m)$. It is clear that $E_0=\Id$. Let us consider $r\geq 1$
and  $x,y\in A$, then
$$
E_r(xy)=D_{r\beta}(xy)=\sum_{\lambda+\sigma=r\beta}
D_\lambda(x)D_\sigma(y)=D_{r\beta}(x)y+xD_{r\beta}(y)+\sum_{
\substack{\lambda +\sigma=r\beta\\ \lambda,\sigma\neq 0}}
D_\lambda(x)D_\sigma(y).
$$
If $[\lambda] \prec [\beta]$, we have $D_\lambda=0$ because
$\lambda\in \s^\Delta_\beta$ and, if $[\beta] \prec [\lambda]$, by
Proposition \ref{Desigualdad 1}, $[\sigma] \prec [\beta]$ so, for
the same reason as before, $D_\sigma=0$. Therefore, the remaining
summands are those for which $[\lambda]=[\sigma]=[\beta]$ and
$$
E_r(xy)=E_r(x)y+xE_r(y)+\sum_{\substack{s\beta+t\beta=r\beta\\
s,t\neq 0}} D_{s\beta}(x)D_{t\beta}(y)=\sum_{s+t=r} E_s(x)E_r(y).
$$
So, we proved that $E\in \HS_k(A;m)$. Let us define
$F:=\psi_{\beta,\Delta} \sbullet E^\ast\in \HS_k^q(A;\Delta)$ and
$D':=F\pcirc D\in \HS_k^q(A;\Delta)$. Hence, $D=F^\ast\pcirc D'=
\left(\psi_{\beta,\Delta} \sbullet E\right)\pcirc D'$, where the
last equality holds since $\psi_{\beta,\Delta}$ has constant
coefficients (this is a very particular case of \cite[Proposition
11]{Na3}) and $(E^\ast)^\ast=E$. It remains to prove the properties
of $D'$.
\medskip

It is clear that $F_\sigma=0$ for all $\sigma\not\in P_\beta^\Delta
\cup\{0\}$ and $F_{r\beta}=E_r^\ast$ for all $r\in \{0,\ldots,m\}$.
Thanks to this, for all $\gamma \in \mathbb N^q$, we have
$$
D'_\gamma=\sum_{\sigma+\lambda=\gamma} F_\sigma \pcirc
D_\lambda=D_\gamma+\sum_{\substack{r\beta+\lambda=\gamma\\ r\neq 0}}
E_r^\ast \pcirc D_\lambda.
$$
Let us assume that $\gamma\in \s_\beta^\Delta$ which implies that
$[\gamma]\prec[\beta]=[r\beta]$ for all $r\neq 0$. By hypothesis,
$D_\gamma=0$. Observe that if $\lambda=0$, then
$[\beta]=[\gamma]\prec [\beta]$ and we have a contradiction, so
$\lambda\neq 0$ and we can apply Proposition \ref{Desigualdad 1}
obtaining that $[\lambda] \prec [\gamma] \prec [\beta]$ and hence,
$\lambda\in \s^\Delta_\beta$. By hypothesis, $D_\lambda=0$ and we
can conclude that $D'_\gamma=0$ for all $\gamma\in \s^\Delta_\beta$.
If $\gamma\in P_\beta^\Delta$, we have $\gamma=t\beta$ for some
$t>0$. From the equality $r\beta+\lambda=t\beta$, we get $\lambda\in
P^\Delta_\beta \cup \{0\}$ and
$$
D'_\gamma=\sum_{r\beta+s\beta=t\beta} E_r^\ast \pcirc
D_{s\beta}=\sum_{r+s=t} E_r^\ast \pcirc E_s=0.
$$
In conclusion, $\supp(D') \subset \{0\} \cup \T^\Delta_\beta
\setminus P^\Delta_\beta$.

Let us assume now that there is $\alpha\in \Delta$ such that
$D_\gamma=0$ for all $\gamma\in P^\Delta_\beta$ with $\gamma \leq
\alpha$ or equivalently, $D_{r\beta}=0$ for all $0<r\beta\leq
\alpha$. Then, $E_r^\ast=0$ for all positive integers $r$ such that
$0<r\beta \leq \alpha$ and, if $\gamma\in \Delta $, $\gamma\leq
\alpha$, we have that
$$
D_\gamma'=\sum_{r\beta+\lambda=\gamma} E_r^\ast \pcirc
D_\lambda=D_\gamma.
$$

To finish the proof we will show the uniqueness. Let us consider
other $T\in\HS_k^q(A;\Delta)$ and $G\in \HS_k(A;m)$ such that
$T_\gamma=0$ for all $\gamma\in \s^\Delta_\beta \sqcup
P^\Delta_\beta$ and
$$
\left(\psi_{\beta,\Delta} \sbullet E\right) \pcirc
D'=D=\left(\psi_{\beta,\Delta} \sbullet G\right)\pcirc T.
$$
From the last equality, we get
$$
H:=\left(\psi_{\beta,\Delta}
\sbullet G^\ast\right)\pcirc \left(\psi_{\beta,\Delta} \sbullet
E\right)=\psi_{\beta,\Delta} \sbullet (G^\ast\pcirc E)=T\pcirc
(D')^\ast
$$
(recall that $\psi_{\beta,\Delta}$ has constant coefficients and see
8. and Proposition 11 of \cite{Na3}). It is easy to see that
$T^\ast_\gamma=(D')_\gamma^\ast=0$ for all $\gamma\in
\s^\Delta_\beta \sqcup P^\Delta_\beta$, so
$T_{r\beta}=(D')^\ast_{r\beta}=0$ for all $r\in \{1,\ldots,m\}$ and
we have that
$$
H_{r\beta}=(G^\ast \pcirc E)_r=(T\pcirc (D')^\ast)_{r\beta}=\sum_{\substack{\lambda+\sigma=r\beta\\ \lambda,\sigma
		\neq 0}} T_\lambda \pcirc (D')^\ast_\sigma.
$$
If $[\lambda] \preceq [\beta]$, then $T_\lambda=0$ because
$\lambda\in \s^\Delta_\beta \sqcup P^\Delta_\beta$ and, if $[\beta]
\prec [\lambda]$, by Proposition \ref{Desigualdad 1}, we get
$[\sigma] \prec[\beta]$ and $(D')^\ast_\sigma=0$. So, $H_{r\beta}=(G^\ast \pcirc E)_r=0$
for all $r\in\{1,\ldots,m\}$. Hence, $G^\ast \pcirc E=\II$ and we
deduce that $G=E$. Now, it is clear that $T=D'$ and we have the
result.	
\end{proof}

In the following theorem, we will prove that any
$\Delta$-variate HS-derivation, where $\Delta$ is a finite
co-ideal, can be decomposed in terms of uni-variate HS-derivations
and substitution maps.

\begin{teo}\label{Descomposicion}
Let us consider a finite co-ideal $\Delta$  and $D\in
\HS_k^q(A;\Delta)$. Let $C:= \#(\C^q_\Delta)$ and
 $\C^q_\Delta = \{\beta^1,\ldots,\beta^C\}$
with $\beta^1 \prec \beta^2 \prec \cdots \prec \beta^C$, and let $m_i
=m^\Delta_{\beta^i}$. Then, there is a unique family $E^i\in
\HS_k(A;m_i)$, $1\leq i\leq C$, such that:
$$ D=\left(\psi_{\beta^1,\Delta} \sbullet E^1\right) \pcirc \left(\psi_{\beta^2,\Delta} \sbullet E^2\right) \pcirc \cdots \pcirc \left(\psi_{\beta^C,\Delta} \sbullet E^C\right).
$$
 Moreover, if for some $a\geq 1$ there is $\alpha\in P^\Delta_{\beta^a}$ such that $D_\gamma=0$ for
all $\gamma\in \s^\Delta_{\beta^a} $ with $\gamma\leq \alpha$, then
$E^{a}_r=D_{r\beta^a}$ for all $r=0,\ldots,
\gcd(\alpha_1,\ldots,\alpha_q)$.
\end{teo}

\begin{proof} We will obtain the $E^i$'s recursively.
Since $\s^\Delta_{\beta^1}=\emptyset$, we can apply Proposition
\ref{Primer paso descomposicion} and we obtain (unique) $E^1\in
\HS_k(A;m_1)$ and $D^1\in \HS^q_k(A; \Delta)$ such that
$$
D=\left(\psi_{\beta^1,\Delta} \sbullet E^1\right)\pcirc D^1
$$
and $D^1_\gamma=0$ for all $\gamma\in
P_{\beta^1}^\Delta=\s^\Delta_{\beta^2}$. Let us assume that for some 
$s\in \mathbb N$, $1\leq s< C$, there 
exist  $E^i\in
\HS_k(A;m_i)$, for $i=1,\ldots, s$, and $D^s\in \HS^q_k(A;
\Delta)$ such that
$$
D= \circ_{i=1}^s \left(\psi_{\beta^i,\Delta} \sbullet E^i\right)
\pcirc D^s
$$
and $D^s_\gamma=0$ for all $\gamma \in \s^\Delta_{\beta^{s+1}}$. If
$s<C-1$, we can apply Proposition \ref{Primer paso descomposicion}
to $D^s$ taking $\beta=\beta^{s+1}$ and we obtain unique $E^{s+1}\in
\HS_k(A;m_{s+1})$ and $D^{s+1}\in \HS_k^q(A;\Delta)$ such that
$D^s=\left(\psi_{\beta^{s+1},\Delta} \sbullet E^{s+1}\right) \pcirc
D^{s+1}$ and $D^{s+1}_\gamma=0$ for all $\gamma \in
\s^\Delta_{\beta^{s+1}} \sqcup
P^\Delta_{\beta^{s+1}}=\s^\Delta_{\beta^{s+2}}$. Hence, we get
$$
D=\circ_{i=1}^{s+1} \left(\psi_{\beta^i,\Delta} \sbullet E^i\right)
\pcirc D^{s+1}.
$$
Let us assume now that $s=C-1$. Let us notice that
$\supp(D^{C-1})\subseteq P^\Delta_{\beta^C} \cup \{0\}$ and we can
write $D^{C-1}=\psi_{\beta^C,\Delta} \sbullet E^C$, where $E^C\in
\HS_k(A;m_C)$ so,
$$
D=\left( \psi_{\beta^1,\Delta} \sbullet E^1\right) \pcirc \cdots
\pcirc \left(\psi_{\beta^C,\Delta} \sbullet E^C\right).
$$
To prove the uniqueness, let us consider another family $F^i\in
\HS_k(A;m_i)$, $1\leq i\leq C$, such that
$$
D= \left(\psi_{\beta^1,\Delta} \sbullet F^1\right) \pcirc \left(
\psi_{\beta^2, \Delta} \sbullet F^2\right) \pcirc \cdots \pcirc
\left(\psi_{\beta^C,\Delta} \sbullet F^C\right).
$$
We denote $T^s=\left(\psi_{\beta^{s+1},\Delta}\sbullet
F^{s+1}\right)\pcirc \cdots \pcirc \left(\psi_{\beta^C,\Delta}
\sbullet F^C\right)\in \HS_k^q(A;\Delta)$ (we put $T^C=\II$). We
will prove that $T^s_\gamma=0$ for all $\gamma \in
\s_{\beta^{s+1}}^\Delta=\s_{\beta^s}^\Delta \sqcup
P^\Delta_{\beta^s}$. Since $\left(\psi_{\beta^i,\Delta} \sbullet
F^i\right)_\lambda=0$ for all $\lambda\not\in P^\Delta_{\beta^i}
\cup \{0\}$, we have that
$$
T^s_\gamma=\sum_{\substack{\lambda_{s+1}+\cdots +\lambda_C=\gamma\\
\lambda_i\in P^\Delta_{\beta^i} \cup \{0\}}}
\left(\psi_{\beta^{s+1},\Delta}\sbullet
F^{s+1}\right)_{\lambda_{s+1}} \pcirc \cdots \pcirc
\left(\psi_{\beta^{C},\Delta}\sbullet F^{C}\right)_{\lambda_{C}}
$$
for all $\gamma\in \Delta$. By Lemma \ref{La suma de vectores}, if
$\gamma\in \s_{\beta^{s+1}}^\Delta$, we have that $[\gamma]\preceq
[\beta^s]\prec[\beta^i]=[\lambda_i]\preceq [\lambda_i+\cdots
+\lambda_C]$, where $i=\min\{ i\in \N_+ \ | \ s+1\leq i\leq C,\
\lambda_i\neq 0\}$. Hence, we can deduce that $T^s_\gamma=0$ for all
$\gamma\in \s_{\beta^{s+1}}^\Delta$.

On the other hand, with the previous notation, we have that
$D=\circ_{i=1}^s \left(\psi_{\beta^i,\Delta} \sbullet
E^i\right)\pcirc D^s$ where $D^s_\gamma=0$ for all $\gamma\in
\s_{\beta^{s+1}}^\Delta$. We will prove that $E^s=F^s$  by induction
on $1\leq s\leq C$. If $s=1$, $D=\left(\psi_{\beta^1,\Delta}
\sbullet E^1\right) \pcirc D^1= \left(\psi_{\beta^1,\Delta}\sbullet
F^1\right)\pcirc T^1$. Thanks to Proposition \ref{Primer paso
descomposicion}, we can deduce that $E^1=F^1$. Let us assume that
$E^i=F^i$ for all $1\leq i<s\leq C$. Then, we have that
$$
D= \circ_{i=1}^{s-1} \left(\psi_{\beta^i,\Delta}\sbullet E^i\right)
\pcirc \left(\psi_{\beta^s,\Delta} \sbullet E^s\right)\pcirc
D^s=\circ_{i=1}^{s-1} \left(\psi_{\beta^i,\Delta}\sbullet E^i\right)
\pcirc \left( \psi_{\beta^s,\Delta} \sbullet F^s\right)\pcirc T^s.
$$
Therefore, $ \left(\psi_{\beta^s,\Delta} \sbullet E^s\right) \pcirc
D^s=\left( \psi_{\beta^s,\Delta}\sbullet F^s\right)\pcirc T^s. $ If
$s=C$,  then it is clear that $E^C=F^C$ ($D^C=\II$) and if $s<C$, we
have that $E^s=F^s$ by Proposition \ref{Primer paso descomposicion}.

Observe that, from the proof of Proposition \ref{Primer paso
descomposicion}, we have that the $r$-component of  $E^s\in
\HS_k(A;m_s)$ is $E^s_r=D^{s-1}_{r\beta^s}$ (we put $D^0=D$). Let us
assume that there is $\alpha\in P^\Delta_{\beta^a}$ such that
$D_\gamma=0$ for all $\gamma \in \s_{\beta^a}^\Delta$ with
$\gamma\leq \alpha$. To see that $E^a_r=D_{r\beta^a}$ for
$r=0,\ldots, \gcd(\alpha_1,\ldots,\alpha_q)$, it is enough to prove
that $D^{a-1}_\gamma=D_\gamma$ for all $\gamma\in \Delta$ with
$\gamma\leq \alpha$ (note that $r\beta^a\leq \alpha$ for all
$r=0,\ldots, \gcd(\alpha_1,\ldots,\alpha_q)$). If $a=1$, then the
result is clear, so  let us assume that $a>1$. We will prove, by
induction on $s=1,\ldots, a-1$, that $D^s_\gamma=D_\gamma$ for all
$\gamma\in \Delta$ with $\gamma\leq \alpha$.

Let us consider $s=1$. Since $\beta^1 \prec\beta^a$, by definition,
$P^\Delta_{\beta^1}\subseteq \s_{\beta^a}^\Delta$. So, $D_\gamma=0$
for all $\gamma\in P^\Delta_{\beta^1}$ with $\gamma\leq \alpha$, and
by Proposition \ref{Primer paso descomposicion},
$D^1_\gamma=D_\gamma$ for all $\gamma\in \Delta $ with $\gamma\leq
\alpha$. Let us assume that, for $s<a-1$, we have that
$D^s_\gamma=D_\gamma$ for all $\gamma\in \Delta $ with $\gamma\leq
\alpha$. In particular, since $\beta^{s+1} \prec \beta^a$,
$D^s_\gamma=0$ for all $\gamma\in P^\Delta_{\beta^{s+1}} \subseteq
\s^{\Delta}_{\beta^a} $ with $\gamma\leq \alpha$.  Recall that
$D^{s+1}$ is obtained applying  Proposition \ref{Primer paso
descomposicion} to $D^s$ with  $\beta=\beta^{s+1}$ so, we deduce
that $D^{s+1}_\gamma=D^s_{\gamma}=D_\gamma$ for all $\gamma\in
\Delta $ with $\gamma\leq \alpha$ and we have the result.
\end{proof}

\begin{cor} Let us consider  a finite co-ideal $\Delta$ and  $D\in \HS_k^q(A;\Delta)$. Let $C:= \#(\C^q_\Delta)$ and $\C^q_\Delta =
\{\beta^1,\beta^s,\ldots, \beta^C\}$ with $\beta^1\prec \beta^2 \prec
\cdots \prec \beta^C$, and let $m_i =m^\Delta_{\beta^i}$. Then,
there is a unique family $E^i\in \HS_k(A;m_i)$, $1\leq i\leq C$,
such that:
$$
D=\psi_\Delta \sbullet \left(E^{1} \boxtimes \cdots \boxtimes
E^{C}\right)
$$
where
$$
\begin{array}{rcccc}
\psi_\Delta:&A\llbracket t_1,\ldots, t_C\rrbracket_{\nabla} &\to &
A\llbracket s_1,\ldots,s_q\rrbracket_\Delta \\
& t_i &\mapsto & s_1^{\beta^i_1}\cdots s_q^{\beta^i_q} & \forall
i=1,\ldots, C
\end{array}
$$
with $\nabla= \{ \gamma\in \N^C \ | \ \gamma \leq
(m_1,\ldots,m_C)\}$.
\end{cor}

\begin{proof}
Let us consider any family $E^i\in \HS_k(A;m_i)$, $1\leq i\leq C$.
Then, it is easy to see that
$$
\psi_{\Delta}\sbullet\left(E^{1} \boxtimes \cdots \boxtimes
E^{C}\right)= \left(\psi_{\beta^1, \Delta} \sbullet E^1\right)
\pcirc \cdots \pcirc \left(\psi_{{\beta^C}, \Delta} \sbullet
E^C\right).
$$
By Theorem \ref{Descomposicion}, there exists a unique family
$E^i\in \HS_k(A;m_i)$ such that $D=\left(\psi_{\beta^1,\Delta}
\sbullet E^1\right) \pcirc \cdots \pcirc \left(\psi_{\beta^C,
\Delta} \sbullet E^C\right)$ and, from the previous equality, we get
$D=\psi_\Delta \sbullet \left(E^1 \boxtimes \cdots \boxtimes E^C
\right)$. If we take another family $F^i\in \HS_k(A;m_i)$, $1\leq
i\leq C$, such that $D=\psi_\Delta \sbullet \left(F^1 \boxtimes
\cdots \boxtimes F^C \right)$. Then, $D=\left(\psi_{\beta^1,\Delta}
\sbullet F^1\right) \pcirc \cdots \pcirc \left(\psi_{\beta^C,
\Delta} \sbullet F^C\right)$ and, by Theorem \ref{Descomposicion},
we deduce that $E^i=F^i$ so, we have the result.
\end{proof}

\begin{ej}{\rm
Let us consider $q=2$, $\Delta=\{ \gamma \in \N^2 \ | \ \gamma \leq
(2,2)\}$ and $D\in \HS_k^2(A;\Delta)$. Then
$\C^2_\Delta=\{\beta^1=(0,1),\beta^2=(1,2),
\beta^3=(1,1),\beta^4=(2,1),\beta^5=(1,0)\}$ and $\beta^1\prec
\cdots \prec \beta^5$. Moreover, it is easy to see that
$m_{\beta^1}^\Delta=m_{\beta^3}^\Delta=m^\Delta_{\beta^5}=2$ and
$m^\Delta_{\beta^2}=m^\Delta_{\beta^4}=1$. We can see $\Delta$ as
follows:

\begin{table}[H]
	\begin{center}
\begin{tabular}{cccccc}
\begin{tikzpicture}[scale=1]
\draw [very thin, gray] (0,0) grid (2,2); \draw[->] (0,0) --
(2.5,0); \draw (2.5,0) node[above]{$\alpha_1$};
 \draw[->] (0,0) -- (0,2.5);
 \draw (0,2.5) node[right]{$\alpha_2$};

 \filldraw[color=red]  (1,0) circle (2pt);
\draw (1,0) node [below]  {$\beta^5$};
 \filldraw[color=red]  (0,1) circle (2pt);
\draw (0,1) node [left]  {$\beta^1$};
 \filldraw[color=red] (1,1) circle (2pt);
\draw (1,1) node [above right]  {$\beta^3$}; 
 \filldraw[color=red] (1,2) circle (2pt);
\draw (1,2) node [above]  {$\beta^2$}; 
 \filldraw[color=red] (2,1) circle (2pt);
\draw (2,1) node [right]  {$\beta^4$};

 \filldraw (2,0) circle (2pt);
 \filldraw (0,2) circle (2pt);
 \filldraw (2,2) circle (2pt);

\draw[color=blue, line width=0.03cm] (0,0) --(0,2);
\end{tikzpicture}
\end{tabular}
\end{center}
\caption{The co-ideal $\Delta$.}
\end{table}

In this picture, the elements of $\Delta$ are represented with a circle that will be red if the element
belongs to $\C^2_\Delta$.  It is clear that the components of $D$ whose index is on the blue line (vertical axis) form a HS-derivation of length 2. In fact, according with the previous theorem, the first step to decompose a $\Delta$-variate HS-derivation is to take that HS-derivation $E^{1}=(\Id, D_{\beta^1}, D_{2\beta^1})=(\Id, D_{(0,1)},D_{(0,2)})\in \HS_k(A;2)$ and the substitution map $\psi_{\beta^1,\Delta}: A\llbracket\mu\rrbracket_2 \ni \mu \mapsto
s_2\in A\llbracket s_1,s_2\rrbracket_\Delta$. Then,
$$
D=\left(\psi_{\beta^1,\Delta} \sbullet E^{1}\right)\pcirc D^1
$$
where $D^1=\left( \psi_{\beta^1,\Delta} \sbullet
\left(E^{1}\right)^\ast\right)\pcirc D$, i.e.
\begin{table}[H]
\begin{center}
\begin{tabular}{cccccc}
\begin{tikzpicture}[scale=1]
\draw [very thin, gray] (0,0) grid (2,2); \draw[->] (0,0) --
(2.5,0); \draw (2.5,0) node[above]{$\alpha_1$};
 \draw[->] (0,0) -- (0,2.5);
 \draw (0,2.5) node[right]{$\alpha_2$};

\draw(1,-1) node{\small{$D$ }};

 \filldraw[color=red]  (1,0) circle (2pt);
\filldraw[color=red]  (0,1) circle (2pt);
\filldraw[color=red] (1,1) circle (2pt);
\filldraw[color=red] (1,2) circle (2pt);
\filldraw[color=red] (2,1) circle (2pt);
\filldraw (2,0) circle (2pt);
\filldraw (0,2) circle (2pt);
\filldraw (2,2) circle (2pt);

\draw[color=blue, line width=0.03cm] (0,0) --(0,2);
 \draw(3,1) node{=};
\end{tikzpicture}

&

\begin{tikzpicture}[scale=1]
\draw [very thin, gray] (0,0) grid (2,2); \draw[->] (0,0) --
(2.5,0); \draw (2.5,0) node[above]{$\alpha_1$};
 \draw[->] (0,0) -- (0,2.5);
 \draw (0,2.5) node[right]{$\alpha_2$};

\draw(1,-1) node{\small{$\psi_{\beta^1,\Delta}\sbullet E^{1}$}};

 \draw (1,0) node[below]  {0};
 \filldraw[color=red] (0,1) circle (2pt);
  \draw (1,1) node[above right]  {0};
  \draw (1,2) node[above]  {0};
  \draw  (2,1) node[right] {0};

  \draw (2,0) node[below]  {0};
 \filldraw (0,2) circle (2pt);
 \draw (2,2) node[above]  {0};

\draw[color=blue, line width=0.03cm] (0,0) --(0,2);

\draw(3,1) node{$\circ$};
\end{tikzpicture}
&
\begin{tikzpicture}[scale=1]
\draw [very thin, gray] (0,0) grid (2,2); \draw[->] (0,0) --
(2.5,0); \draw (2.5,0) node[above]{$\alpha_1$};
 \draw[->] (0,0) -- (0,2.5);
 \draw (0,2.5) node[right]{$\alpha_2$};

\draw(1,-1) node{\small{$D^1$}};

 \filldraw[color=red]  (1,0) circle (2pt);
\draw (0,1) node[left] {0};
\filldraw[color=red] (1,1) circle (2pt);
\filldraw[color=red] (1,2) circle (2pt);
\filldraw[color=red] (2,1) circle (2pt);

\filldraw (2,0) circle (2pt);
\draw (0,2) node[left] {$0$};
\filldraw (2,2) circle (2pt);

 \draw[color=blue] (0,0) --(1,2);
\end{tikzpicture}

\end{tabular}
\end{center}
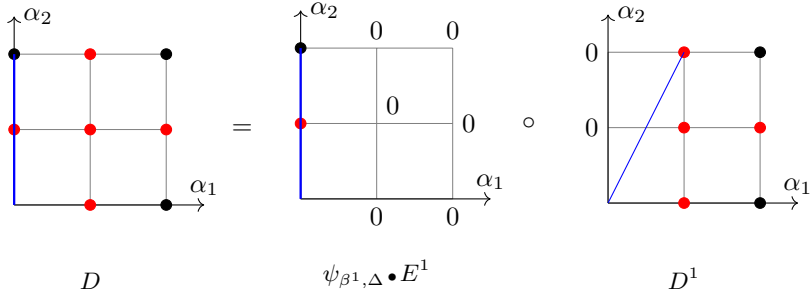
\caption{First step of the decomposition of $D$.}
\end{table}
If we continue with the steps of the proof of the theorem, we have
to decompose $D^1$ using Proposition \ref{Primer paso
descomposicion}. Since $D^1_\gamma=0$ for all $\gamma \in
P^\Delta_{\beta^1}=\s^\Delta_{\beta^2}=\{(0,1),(0,2)\}$, we have
that $E^2=(\Id, D^1_{\beta^2})= (\Id,D_{(1,2)}-D_{(0,1)}D_{(1,1)}-
D_{(0,2)}D_{(1,0)}+D^2_{(0,1)}D_{(1,0)})\in \HS_k(A;1)\in
\HS_k(A;1)$ (blue line in $D^1$). Now, we can decompose $D^1$ as
$$
D^1=\left(\psi_{\beta^2,\Delta} \sbullet E^2\right) \pcirc D^2,
$$
where $\psi_{\beta^2,\Delta}: A\llbracket\mu\rrbracket_1 \ni
\mu\mapsto s_1s_2^2\in A\llbracket s_1,s_2\rrbracket_\Delta$ and
$D^2=\left(\psi_{\beta^2,\Delta} \sbullet (E^2)^\ast\right)\pcirc
D^1= \left(\psi_{\beta^2,\Delta} \sbullet (E^2)^\ast\right)\pcirc
\left(\psi_{\beta^1,\Delta} \sbullet (E^1)^\ast\right) \pcirc D\in
\HS^2(A;\Delta)$ with $D^2_\gamma=0$ for all $\gamma \in
\s^\Delta_{\beta^2} \sqcup
P^\Delta_{\beta^2}=\{(0,1),(0,2),(1,2)\}$: 
\begin{table}[H]
\begin{center}
	\begin{tabular}{cccccc}
		\begin{tikzpicture}[scale=1]
\draw [very thin, gray] (0,0) grid (2,2); \draw[->] (0,0) --
	(2.5,0); \draw (2.5,0) node[above]{$\alpha_1$};
	\draw[->] (0,0) -- (0,2.5);
\draw (0,2.5) node[right]{$\alpha_2$};
				
\draw(1,-1) node{\small{$D^1$ }};
				
\filldraw[color=red]  (1,0) circle (2pt);
\draw (0,1) node[left]{$0$}; 
\filldraw[color=red] (1,1) circle (2pt);
 \filldraw[color=red] (1,2) circle (2pt);
\filldraw[color=red] (2,1) circle (2pt);

\filldraw (2,0) circle (2pt);
\draw (0,2) node[left] {0};
\filldraw (2,2) circle (2pt);

\draw[color=blue] (0,0) --(1,2);
\draw(3,1) node{=};
\end{tikzpicture}
			
&
			
\begin{tikzpicture}[scale=1]
\draw [very thin, gray] (0,0) grid (2,2); \draw[->] (0,0) --
(2.5,0); \draw (2.5,0) node[above]{$\alpha_1$};
\draw[->] (0,0) -- (0,2.5);
\draw (0,2.5) node[right]{$\alpha_2$};

\draw(1,-1) node{\small{$\psi_{\beta^2,\Delta}\sbullet E^{2}$}};

\draw (1,0) node[below]  {0};
\draw (0,1) node[left] {0};
\draw (1,1) node[above right]  {0};
\draw  (2,1) node[right] {0};
\draw (2,0) node[below]  {0};				
\filldraw[color=red] (1,2) circle (2pt);
\draw (2,2) node[above]  {0};
\draw (0,2) node[left] {0};

\draw[color=blue] (0,0) --(1,2);			
				
\draw(3,1) node{$\circ$};
\end{tikzpicture}		
&
\begin{tikzpicture}[scale=1]				
\draw [very thin, gray] (0,0) grid (2,2); \draw[->] (0,0) --
		(2.5,0); \draw (2.5,0) node[above]{$\alpha_1$};
				\draw[->] (0,0) -- (0,2.5);
				\draw (0,2.5) node[right]{$\alpha_2$};

				\draw(1,-1) node{\small{$D^2$}};

\filldraw[color=red]  (1,0) circle (2pt);
\draw (0,1) node[left] {0};
\filldraw[color=red] (1,1) circle (2pt);
\draw (1,2) node[above] {0};
\filldraw[color=red] (2,1) circle (2pt);
				
\filldraw (2,0) circle (2pt);
\draw (0,2) node[left] {0};
\filldraw(2,2) circle (2pt);

				\draw[color=blue] (0,0) --(2,2);
			\end{tikzpicture}
			
		\end{tabular}
	\end{center}
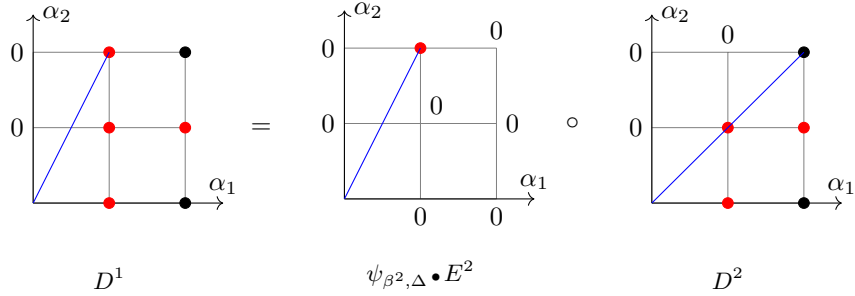
	\caption{Second step of the decomposition of $D$.}
\end{table}

Hence,
$$D=\left(\psi_{\beta^1,\Delta} \sbullet E^1\right) \pcirc
\left(\psi_{\beta^2,\Delta}\sbullet E^2\right)\pcirc D^2.
$$

If we continue with the process described in the proof of the
previous theorem, we can find the decomposition of $D$. In this
case,
$$
D=\left(\psi_{\beta^1,\Delta} \sbullet E^{1}\right) \pcirc
\left(\psi_{\beta^2,\Delta} \sbullet E^{2}\right) \pcirc
\left(\psi_{\beta^3,\Delta} \sbullet E^{3}\right) \pcirc
\left(\psi_{\beta^4,\Delta} \sbullet E^{4}\right) \pcirc
\left(\psi_{\beta^5,\Delta} \sbullet E^{5}\right)
$$
where $E^{3}=(\Id,D^2_{\beta^3},
D^2_{2\beta^3})=(\Id,D_{(1,1)}-D_{(0,1)}D_{(1,0)}, E^3_2)\in
\HS_k^2(A;2), \ E^{4}=(\Id, D_{\beta^4}^3)=(\Id,
D_{(2,1)}-D_{(0,1)}D_{(2,0)}-D_{(1,1)}D_{(1,0)}+D_{(0,1)}D^2_{(1,0)})\in
\HS_k(A;1) $ and $E^{5}=(\Id,D^4_{\beta^5},
D^4_{2\beta^5})=(\Id,D_{(1,0)},D_{(2,0)})\in \HS_k(A;2)$ (let us
notice that the components of $E^i$ are those whose indices are on
the line through $\beta^i$ and $(0,0)$ in the graphical
representation of $\Delta$) with
$$
E^3_2=D_{(2,2)}-D_{(0,1)}D_{(2,1)}-D_{(0,2)}D_{(2,0)}-D_{(1,2)}D_{(1,0)}+
D^2_{(0,1)}D_{(2,0)}+D_{(0,1)}D_{(1,1)}D_{(1,0)}
+D_{(0,2)}D_{(1,0)}^2- D_{(0,1)}^2D_{(1,0)}^2,$$
$D^i=\left(\circ_{j=i}^1 \psi_{\beta^j, \Delta} \sbullet
(E^j)^\ast\right)\pcirc D$ for $i=3,4$  and the substitution maps
are $ \psi_{\beta^3,\Delta}: A\llbracket\mu\rrbracket_2\ni
\mu\mapsto s_1s_2\in A\llbracket s_1,s_2\rrbracket_\Delta; \
\psi_{\beta^4,\Delta}: A\llbracket\mu\rrbracket_1\ni \mu\mapsto
s_1^2s_2\in A\llbracket s_1,s_2\rrbracket_\Delta $ and $
\psi_{\beta^5,\Delta}: A\llbracket\mu\rrbracket_2\ni \mu \mapsto
s_1\in A\llbracket s_1,s_2\rrbracket_\Delta$.}
\end{ej}

The next corollary provides a way of dealing with infinite co-ideals
via finite approximations (in the sense of Definition \ref{def:infinite-comp}).

\begin{cor}\label{Descomposicion infinita}
Let us consider  a co-ideal $\Delta\subseteq \N^q$ and $D\in
\HS^q_k(A; \Delta)$. Let us denote $m_\beta=m^\Delta_\beta$ for
$\beta \in \C^q_\Delta$. Then, there exists a unique family
$E^\beta\in \HS_k(A;m_\beta)$, for $\beta \in \C^q_\Delta$, such
that the family $\psi_{\beta, \Delta} \sbullet E^\beta$, $\beta \in
\C^q_\Delta$, is composable (see Definition \ref{def:infinite-comp})
and
$$
D= \circ_{\beta \in \C^q_\Delta} \left( \psi_{\beta, \Delta}
\sbullet E^\beta \right).
$$
Moreover, if there is $\alpha\in P^\Delta_\beta$ for some $\beta\in
\C^q_\Delta$ such that $D_\gamma=0$ for all $\gamma\in
\s^\Delta_{\beta} $ with $\gamma\leq \alpha$, then
$E^{\beta}_n=D_{n\beta}$ for all $n=0,\ldots,
\gcd(\alpha_1,\ldots,\alpha_q)$.
\end{cor}

\begin{proof}
Let us consider the finite co-ideals\footnote{Actually, we could
consider any increasing exhaustive sequence of finite co-ideals
contained in $\Delta$.}  $\Delta^r:=\Delta \cap \{\alpha\in \N^q \ |
\ |\alpha|\leq r\}$. We have $\Delta^r\subseteq \Delta^{r+1}$ for
all $r\geq 1$ and $\Delta=\bigcup_r \Delta^r$. Moreover, if
$\nabla'\subseteq \nabla$ are two non-empty co-ideals, for all
$\beta\in \C^q_\nabla$, the substitution map $\tau_{\nabla\nabla'}
\pcirc \psi_{\beta,\nabla}:
 A\llbracket\mu\rrbracket_{m_\beta^\nabla} \ni \mu \to s_1^{\beta_1}\cdots
s_q^{\beta_q}\in A\llbracket s_1,\ldots,s_q\rrbracket_{\nabla'}$ is
\begin{equation}\label{truncacion psi}
\tau_{\nabla\nabla'}\pcirc
\psi_{\beta,\nabla}=\left\{\begin{array}{ll}
0&\mbox{ if } \beta\not\in \C^q_{\nabla'}\\
\psi_{\beta,\nabla'} \pcirc \tau_{m_\beta^\nabla m_\beta^{\nabla'}}
& \mbox{ if } \beta\in \C^q_{\nabla'}.
\end{array}\right.
\end{equation}

We denote $D^r:= \tau_{\Delta\Delta^r}(D)\in \HS_k^q(A;\Delta^r)$,
$\C^q_r:=\C^q_{\Delta^r}=\{\beta^{1,(r)} \prec \beta^{2,(r)}
\prec\cdots \prec \beta^{C_r,(r)}\}$ and
$m^{(r)}_{i}=m^{\Delta^r}_{\beta^{i,(r)}}$. It is clear that
$\C^q_r\subseteq \C^q_{r+1}$ for all $r\geq 1$. Moreover, for all
$\beta\in \C^q$, there exists $b_\beta\geq 1$ such that $\beta\in
\Delta^r$ for all $r\geq b_\beta$ and $\beta\not\in
\Delta^{b_\beta-1}$. Hence, we have that
$\beta=\beta^{i_{r,\beta},(r)}$ for all $r\geq b_\beta$ and the
chain
$$m^{(b_{\beta})}_{i_{b_\beta,\beta}}\leq \cdots \leq  m^{(r)}_{i_{r,\beta}} \leq m^{(r+1)}_{i_{r+1,\beta}}\leq \cdots \leq m^\Delta_\beta.
$$
Observe that if $m^\Delta_\beta<\infty$, then there exists $n\geq
b_\beta$ such that $m^{(n)}_{i_{n,\beta}}=m^\Delta_\beta$. For all
$r\geq 1$, by Theorem \ref{Descomposicion}, there exists a unique
family $E^{j,(r)}\in \HS_k(A;m^{(r)}_j)$ such that
$$
D^r=\left(\psi_{\beta^{1,(r)},\Delta^r} \sbullet
E^{1,(r)}\right)\pcirc \cdots \pcirc
\left(\psi_{\beta^{C_r,(r)},\Delta^r} \sbullet E^{C_r,(r)}\right).
$$
Since $\tau_{\Delta^{r+1}\Delta^r}(D^{r+1})=D^r$ and
(\ref{truncacion psi}), we have that
$$
\begin{array}{rll}
D^r=&\left(\left(\tau_{\Delta^{r+1}\Delta^r} \pcirc\psi_{\beta^{1,(r+1)},\Delta^{r+1}}\right) \sbullet E^{1,(r+1)}\right)\pcirc \cdots \pcirc \left(\left(\tau_{\Delta^{r+1}\Delta^r} \pcirc\psi_{\beta^{C_{r+1},(r+1)},\Delta^{r+1}} \right)\sbullet E^{C_{r+1},(r+1)}\right)\\
=&\left(\psi_{\beta^{1,(r)},\Delta^r} \sbullet F^1\right) \pcirc
\cdots \pcirc \left(\psi_{\beta^{C_r,(r),\Delta^r}} \sbullet
F^{C_r}\right),
\end{array}
$$
with $F^j:=\tau_{m^{(r+1)}_{i_{r+1,\beta}}
m^{(r)}_{j}} ( E^{i_{r+1,\beta},(r+1)})$ for $\beta=\beta^{j,(r)}=\beta^{i_{r+1,\beta},(r+1)}\in
\C^q_r$.
Thanks to the uniqueness,
we obtain that $F^j=E^{j,(r)}$ for all $j$. Hence, for all $\beta\in
\C^q_{\Delta}$, we have a set $\{E^{i_{r,\beta},(r)}\in
\HS_k(A;m^{(r)}_{i_{r,\beta}})\}_{r\geq b_\beta}$ such that
$\tau_{m^{(r+1)}_{i_{r+1,\beta}}m^{(r)}_{i_{r,\beta}} }
(E^{i_{r+1,\beta},(r+1)})=E^{i_{r,\beta},(r)}$. Then, we define
$$
E^\beta=\lim_{\stackrel{\longleftarrow} {\substack{r\geq b_\beta}}}
E^{i_{r,\beta},(r)}\in \HS_k(A;m^\Delta_\beta).
$$
The family $\{\psi_{\beta,\Delta} \sbullet E^\beta\}_{\beta\in
\C^q_\Delta}$ is composable since,  for any finite non-empty
co-ideal $\nabla\subseteq \Delta$, the set $\C^q_\nabla$ is finite
and, thanks to (\ref{truncacion psi}), $(\tau_{\Delta\nabla} \pcirc
\psi_{\beta,\Delta} ) \sbullet E^\beta=\II$ for all $\beta\not\in
\C^q_\nabla$. To prove that $D= \circ_{\beta\in \C^q_\Delta}
\left(\psi_{\beta,\Delta}\sbullet E^\beta\right)$, we have to see
that, for all finite co-ideal $\nabla\subseteq \Delta$,
$$
\tau_{\Delta\nabla}(D)=\tau_{\Delta\nabla} \left(\circ_{\beta\in
\C^q_\Delta} (\psi_{\beta,\Delta} \sbullet
E^\beta)\right)=\circ_{\beta\in \C^q_\Delta} \left(
\left(\tau_{\Delta\nabla} \pcirc \psi_{\beta,\Delta}\right) \sbullet
E^\beta\right)=\circ_{\beta\in \C^q_\nabla}\left(
\left(\psi_{\beta,\nabla}\pcirc \tau_{m_\beta^\Delta
m_\beta^\nabla}\right) \sbullet E^\beta\right) .$$ So, let us
consider a finite co-ideal $\nabla\subseteq \Delta$. Then, there
exists $r\geq 1$ such that $\nabla \subseteq \Delta^r$ and
$\tau_{\Delta\nabla}(D)=\tau_{\Delta^r\nabla}(D^r)$. Thanks to
(\ref{truncacion psi}), we have
$$
\tau_{\Delta\nabla}(D)=\left(\tau_{\Delta^r\nabla} \pcirc
\psi_{\beta^{1,(r)},\Delta^r} \sbullet E^{1,(r)}\right) \pcirc
\cdots \pcirc \left(\tau_{\Delta^r\nabla} \pcirc
\psi_{\beta^{C_r,(r)},\Delta^r} \sbullet
E^{C_r,(r)}\right)=\circ_{\beta\in \C^q_\nabla} \left(
\psi_{\beta,\nabla} \sbullet G^\beta\right),
$$
where $G^\beta:=\tau_{m^{(r)}_{i_{r,\beta}}
m^\nabla_\beta}(E^{i_{r,\beta},(r)})=\tau_{m^\Delta_\beta
m^\nabla_\beta}(E^\beta)$ for all $\beta=\beta^{i_{r,\beta},(r)}\in
\C^q_\nabla \subseteq\C^q_r$. Hence, we have the equality.

The family $E^\beta$, ${\beta\in \C^q_\Delta}$, is unique: let
$H^\beta\in\HS_k(A;m_\beta^\Delta)$, ${\beta\in \C^q_\Delta}$, be
another family such that $D=\circ_{\beta\in \C^q_\Delta} \left(
\psi_{\beta,\Delta}\sbullet H^\beta\right)$. From (\ref{truncacion
psi}),
$$D^r=\tau_{\Delta\Delta^r}\left( \circ_{\beta\in \C^q_\Delta} \psi_{\beta,\Delta} \sbullet E^\beta\right)
=\circ_{\beta\in \C^q_{\Delta^r}}
\left(\psi_{\beta,\Delta^r}\sbullet \left(\tau_{m_\beta^\Delta
m_\beta^{\Delta^r}}\left(E^\beta\right)\right)\right).
$$
and doing a similar computation, $D^r=\circ_{\beta\in
\C^q_{\Delta^r}} \left(\psi_{\beta,\Delta^r}\sbullet
\left(\tau_{m_\beta^\Delta
m_\beta^{\Delta^r}}(H^\beta)\right)\right)$. From the uniqueness of
Theorem \ref{Descomposicion}, we deduce that, for all $r\geq
b_\beta$, $\tau_{m_\beta^\Delta
m_\beta^{\Delta^r}}(E^\beta)=\tau_{m_\beta^\Delta
m_\beta^{\Delta^r}}(H^\beta)$ and so $E^\beta=H^\beta$.
\medskip

Let us assume now that $\alpha\in P^\Delta_\beta$ for some $\beta\in
\C^q_\Delta$ such that $D_\gamma=0$ for all $\gamma\in
\s^\Delta_{\beta} $ with $\gamma\leq \alpha$. Let us consider $r\geq
b_\beta$ such that $\gcd(\alpha_1,\ldots,\alpha_q)\leq
m_{i_{r,\beta}}^{(r)}$ (for example, $r=|\alpha|$). Then,
$D^r_\gamma=0$ for all $\gamma\in \s_\beta^\Delta \cap
\Delta^r=\s_\beta^{\Delta^r}$ with $\gamma\leq \alpha$. Then, since
$\tau_{m^{\Delta}_{\beta}m^{(r)}_{i_{r,\beta}}}(E^\beta)=
E^{i_{r,\beta},(r)}\in \HS_k^q(A;\Delta^r)$, by Theorem
\ref{Descomposicion},
$E^{\beta}_n=E^{i_{r,\beta},(r)}_n=D^r_{n\beta}=D_{n\beta}$ for all
$n=0,\ldots, \gcd(\alpha_1,\ldots,\alpha_q)$.
\end{proof}

\begin{cor}
Let $k$ be a ring of positive prime characteristic $p>0$, $\Delta$ a co-ideal,
$\alpha\in \C^q_{\Delta}$ and $d,s\geq 1$ such that $dp^s\alpha\in
\Delta$. Let us consider $D\in \HS_k^q(A;\Delta)$ such that
$D_\gamma=0$ for all $\gamma\in \s^\Delta_{\alpha}$ with $\gamma\leq
d\alpha$. If $D_{r\alpha}=0$ for all $r=1,\ldots, d-1$ then,
$D_{d\alpha}$ is a $p^s$-integrable derivation.
\end{cor}

\begin{proof}
By Corollary \ref{Descomposicion infinita}, there exists
$E^{\alpha}\in \HS_k(A;m_{\alpha}^\Delta)$ such that
$E_r^{\alpha}=D_{r\alpha}$ for all $r=1,\ldots,
\gcd(d\alpha_1,\ldots,d\alpha_q)=d$. Since $dp^s\leq
m^\Delta_{\alpha}$, we can consider $E=\tau_{m^\Delta_\alpha \
dp^s}(E^\alpha)\in \HS_k(A;dp^s)$ such that $E_r=0$ for all
$r=1,\ldots,d-1$ and $E_d=D_{d\alpha}$. By Proposition \ref{Prop 3.14}, we can deduce that $D_{d\alpha}$ is a $p^s$-integrable
derivation.
\end{proof}

Let us recall that a {\it Lie-Rinehart algebra} $L$ over $A/k$ (see
\cite{rine_63}) is a left $A$-module and a $k$-Lie algebra endowed
with an ``anchor'' map $ \varrho: L \longrightarrow \Der_k(A)$ which
is $A$-linear, a map of $k$-Lie algebras and the following
compatibility holds:
$$
[\lambda, a\lambda'] = a [\lambda, \lambda'] + \varrho(\lambda)(a)
\lambda',\ \forall \lambda, \lambda'\in L, \forall a \in A.
$$
We usually write $\lambda(a)$ for $\varrho(\lambda)(a)$. Moreover,
if $k$ has positive prime characteristic $p>0$, a Lie-Rinehart algebra $L$
is called {\it restricted} if $L$ is a restricted Lie algebra (see \cite[Chap. V, \S 7]{jac_79}) such that
$$
(a\lambda)^{[p]}=a^p\lambda^{[p]}+(a\lambda)^{p-1}(a)\lambda \ \
\forall \lambda \in L, \ \forall a\in A
$$
(see \cite{rum} for more information about restricted Lie-Rinehart ($\equiv$ Lie algebroids)).
\medskip

Thanks to Corollary \ref{El conmutador es integrable},
modules $\IDer_k(A;m)$, $m\in \mathbb N\cup \{\infty\}$, and $\IDer_k^f(A)$
will be Lie-Rinehart algebras, the anchor maps being the inclusions in
$\Der_k(A)$. Moreover, if $k$ has positive prime characteristic and $m$ is
a positive integer, $\IDer_k(A;m)$ and $\IDer_k^f(A)$ will be restricted
by Proposition \ref{Elevar a p}. 

\begin{cor}\label{El conmutador es integrable}
Let $\delta,\varepsilon\in \IDer_k(A;m)$ be $m$-integrable
derivations, for $m\in \N\cup \{\infty\}$. Then the bracket
$[\delta,\varepsilon]=\delta \varepsilon -\varepsilon\delta$ is also
$m$-integrable.
\end{cor}

\begin{proof}
Let us consider $D,E\in \HS_k(A;m)$ $m$-integrals of $\delta,\varepsilon$ respectively and let us denote
$$
F:=(D\boxtimes E)\pcirc (D^\ast \boxtimes E^\ast)\in
\HS^2_k(A;\Delta),
$$
where $\Delta=\{\beta\in \N^2\ | \ \beta\leq (m,m)\}$ if $m\in \N$
and $\Delta=\N^2$ if $m=\infty$. We have that $F_{(0,1)}=0$ and,
since $E^\ast_1=-E_1=-\varepsilon$ and $D^\ast_1=-D_1=-\delta$, we
get
$$
F_{(1,1)}=D_1E_1+D_1E_1^\ast+E_1D_1^\ast+D_1^\ast
E_1^\ast=[D_1,E_1]=[\delta,\varepsilon].
$$
Let us consider $\alpha=(1,1)\in \C^2_\Delta$. Then, $m^\Delta_{\alpha}=m$ and
$\s_{\alpha}^{\Delta} \cap \{ \lambda\in \N^q_+\ |\ \lambda \leq
\alpha\}=\{(0,1)\}$. By Corollary \ref{Descomposicion infinita},
there exists $E^\alpha\in \HS_k(A;m)$ such that
$E^{\alpha}_1=F_{(1,1)}=[\delta,\varepsilon]$, and so $[\delta,\varepsilon]$ is an
$m$-integrable derivation.
\end{proof}

\begin{prop}\label{Elevar a p}
Let  $k$ be a ring of positive prime characteristic $p>0$ and $m\in \mathbb N$. If
$\delta\in \IDer_k(A;m)$, then $\delta^p\in \IDer_k(A;m)$, and so $\IDer_k(A;m)$ is a restricted Lie-Rinehart algebra.
\end{prop}

\begin{proof}
By Theorem 4.1 from \cite{Ti1}, we only have to prove the result for
powers of $p$ since $\IDer_k(A;m)=\IDer_k(A;p^\alpha)$ for
$\alpha=\max \{\tau \in \mathbb N_+\ | \ p^\tau\leq m\}$. So, let us
consider $\delta \in \IDer_k(A;p^\alpha)=\IDer_k(p^{\alpha+1}-1)$
for some $\alpha \geq 1$ and $D\in \HS_k(A; p^{\alpha+1}-1)$ a
$(p^{\alpha+1}-1)$-integral of $\delta$, and denote $E=D^p\in \HS_k(A;p^{\alpha+1}-1)$.
\medskip

The $n$-component of $E$, for $1\leq n < p^{\alpha+1}$, is $E_n=\sum_{|i|=n} D_{i_1} \pcirc D_{i_2}
\pcirc \cdots \pcirc D_{i_p}$, with $i\in\mathbb{N}^p$ and $|i|=i_1+\cdots+i_p$. We have:
\begin{equation} \label{eq:En}
E_n =\cdots = \sum_{\substack{H\subset \{1,\dots,p\}\\ H\neq \emptyset}} \sum_{\supp(i)=H} D_{i_1} \pcirc D_{i_2}
\pcirc \cdots \pcirc D_{i_p} = \sum_{k=1}^p \sum_{\substack {a_1+\cdots+a_k=n\\ a_j > 0}}
\binom{p}{k} D_{a_1}\pcirc \cdots \pcirc  D_{a_k},
\end{equation}
with $\supp(i) = \{j\in \{1,\dots,p\}\ |\ i_j\neq 0\}$. 
If $D$ was $p^{\alpha+1}$-integrable, i.e. if it had an extension up to a Hasse--Schmidt derivation of length $p^{\alpha+1}$, that we also call $D$, the expression \eqref{eq:En} would hold for $n=p^{\alpha+1}$, but
\begin{equation} \label{eq:Epa+1}
E_{p^{\alpha+1}} = \cdots = pD_{p^{\alpha+1}}+\sum_{\substack{|i|=p^{\alpha+1}\\ i_j< p^{\alpha+1}}}
D_{i_1}  \pcirc \cdots \pcirc D_{i_p}
 =\sum_{\substack{|i|=p^{\alpha+1}\\ i_j< p^{\alpha+1}}} D_{i_1}
\pcirc \cdots \pcirc D_{i_p}
\end{equation}
would not depend on $D_{p^{\alpha+1}}$. With this idea in mind, we define $E_{p^{\alpha+1}}$ as in equation \eqref{eq:Epa+1} and we can prove directly that the resulting sequence $(\Id,E_1,\dots,E_{p^{\alpha+1}-1},E_{p^{\alpha+1}})$ is a Hasse--Schmidt derivation of length $p^{\alpha+1}$.

Now, from equation \eqref{eq:En} we deduce that $E_n=0$ for all $1\leq n<p$ and
$E_p=D_1^p=\delta^p$, and by Proposition \ref{Prop 3.14}, we conclude that $\delta^p\in
\IDer_k(A;p^\alpha)$.
\end{proof}

\begin{nota} An obvious consequence of Proposition \ref{Elevar a p} is that $\delta^p\in \IDer^f_k(A)$ whenever $\delta\in \IDer^f_k(A)$, and so $\IDer^f_k(A)$ is a restricted Lie-Rinehart algebra. However, we do not know whether the same result holds for $\IDer_k(A;\infty)$ instead of $\IDer^f_k(A)$.
\end{nota}

\begin{ej}
Let us consider $\delta,\varepsilon\in \IDer_k(A;4)$ and $D,E\in
\HS_k(A;4)$ a $4$-integral of $\delta$ and $\varepsilon$
respectively. Then, we define $F=(D \boxtimes E) \pcirc (D^\ast
\boxtimes E^\ast)$. Following the steps of the proof we get that a
4-integral of $[\delta,\varepsilon]$ is
$(\Id,[\delta,\varepsilon],H_2,H_3,H_4)\in \HS_k(A;4)$:
$$
H_2= F_{(2,2)}=D_2E_2+D_1E_2D_1^\ast+E_2D_2^\ast+
(D_2E_1+D_1E_1D_1^\ast+E_1D_2^\ast)E_1^\ast,
$$
\begin{eqnarray*}
& \displaystyle H_3=F_{(3,3)}-F_{(1,2)}F_{(2,1)} =
\\
& \displaystyle=\sum_{i+j=3} E_i D_3^\ast E^\ast_j
+D_2\left(\sum_{i+j=3} E_iD_1^\ast
E_j^\ast\right)+D_1\left(\sum_{i+j=3} E_iD_2^\ast E_j^\ast\right) -
\left( \sum_{i+j=2} E_iD_1^\ast E_j\right) \left(\sum_{i+j=2}
D_iE_1D_j^\ast\right)
\end{eqnarray*}
and
$H_4=F_{(4,4)}-F_{(1,3)}F_{(3,1)}-F_{(1,2)}F_{(3,2)}-F_{(2,3)}F_{(2,1)}+F_{(1,2)}[\delta,\varepsilon]F_{(2,1)}
$ with
$$
F_{(4,4)}=\sum_{r=1}^4\left(\sum_{i+j=4}
D_iE_rD^\ast_j\right)E^\ast_{4-r}, \  \ F_{(1,r)}=\sum_{i+j=r} E_i
D_1^\ast E_j, \ \ F_{(r,1)}=\sum_{i+j=r} D_iE_1D_j^\ast
$$
$$
F_{(2,3)}=\sum_{i+j=3} E_i D_2^\ast E^\ast_j+D_1 \left(\sum_{i+j=3}
E_iD_1^\ast E^\ast_j\right) \text{ and } \
F_{(3,2)}=\sum_{i+j=3}D_iE_2D_j^\ast
+\left(\sum_{i+j=3}D_iE_1D_j^\ast\right)E^\ast_1.
$$
\end{ej}

\section{Poisson structures}

In \cite{nar_2009}, the first author has introduced a canonical map
of graded $A$-algebras $\vartheta^\infty: \Gamma_A \IDer_k(A;\infty)
\longrightarrow \gr \DD_{A/k}$, where $\DD_{A/k}$ is the filtered
ring of linear differential operators of $A$ over $k$ and $\Gamma_A$ denotes the {\em divided power algebra} functor. It is
determined in the following way. For each $\infty$-integrable
derivation $\delta \in \IDer_k(A;\infty) $ let us choose an integral
$D=(\Id,D_1=\delta,\dots)\in \HS_k(A;\infty)$. Then the symbol
$\sigma_n(D_n)$ does not depend on the choice of $D$ and $
\vartheta^\infty(\gamma_n(\delta)) = \sigma_n(D_n)$.
\medskip

Actually, the above construction also works if we take
$\IDer^f_k(A)$ instead of $\IDer_k(A;\infty)$ and we obtain a unique
map of graded $A$-algebras
$$
\vartheta^f: \Gamma_A \IDer^f_k(A) \longrightarrow \gr \DD_{A/k}
$$
determined in a similar way: for each $f$-integrable derivation
$\delta \in \IDer^f_k(A) $ and for each $n\geq 1$, let us choose a
$n$-integral $D=(\Id,D_1=\delta,\dots,D_n)\in \HS_k(A;n)$. Then the
symbol $\sigma_n(D_n)$ only depends on $\delta$ and  not on the
choice of $D$, and $ \vartheta^f(\gamma_n(\delta)) = \sigma_n(D_n)
$. Clearly, $\vartheta^f$ is an extension of $\vartheta^\infty$.
\medskip

On the other hand, since the ring of differential operators
$\DD_{A/k}$ is filtered with commutative graded ring, we know that
its graded ring $\gr \DD_{A/k}$ has a canonical Poisson bracket
given by (cf. \cite{gab_81}):
$$ \{\sigma_d(P),\sigma_e(Q)\} = \sigma_{d+e-1}([P,Q])
$$
for all $P\in \DD^d_{A/k}$ and all $Q\in \DD^e_{A/k}$, where
$\sigma_d: \DD^d_{A/k} \to \gr^d \DD_{A/k}$ is the $d$-symbol map.
It is an skew-symmetric $k$-biderivation and satisfies Jacobi
identity, and so $\gr \DD_{A/k}$ becomes a Poisson algebra.
Moreover, this Poisson bracket is graded of degree $-1$.
\medskip

The goal of this section is, by using the fact that $\IDer^f_k(A)$ and $\IDer_k(A;\infty)$ are Lie-Rinehart algebras (see Corollary \ref{El conmutador es integrable}), to exhibit natural Poisson algebra structures on
$\Gamma_A \IDer^f_k(A)$ and $\Gamma_A \IDer_k(A;\infty)$ in such a
way that $\vartheta^\infty$ and $\vartheta^f$ becomes maps of
Poisson algebras.
\medskip

Let us recall that, for any $A$-module $M$, its {\em
divided power algebra} $\Gamma_A M$, endowed with the power divided
maps $\gamma_n:M \to  \Gamma_A^n M$, $n\geq 0$, has been defined in
\cite[Chap. III, 1]{roby_63} (see also \cite[App. A]{bert_ogus}). It
is a graded commutative $A$-algebra $\Gamma_A M = \bigoplus_{n\geq
0} \Gamma^n_A M$, with $\Gamma^0_A M = A$, $\Gamma^1_A M = M$ and
$\Gamma^n_A M$ is generated as $A$-module by the $\gamma_n(x)$,
$x\in M$, and it has some universal property that we will not detail
here (see \cite[Th. III.1]{roby_63}). When $\Q \subset A$, then
$\Gamma_A M$ coincides with the symmetric algebra $\Sym_A M$ and
$\gamma_n(x) = \frac{x^n}{n!}$ for all $x\in M$ and all $n\geq 0$.
\medskip

First, let us see the following general result.

\begin{prop} \label{prop:poisson_Gamma}
If $L$ is a Lie-Rinehart algebra over $A/k$, then there is a unique Poisson structure
$\{-,-\}$ on $\Gamma_A L$ such that:
\begin{enumerate}
\item[(i)] $\{a,a'\}=0$ for all $a,a'\in A$.
\item[(ii)] $\{\gamma_m(\lambda),a\} = \lambda(a)\, \gamma_{m-1}(\lambda)$ for all $\lambda\in L$, all $a\in A$ and all $m\geq 1$.
\item[(iii)] $
\{\gamma_m(\lambda),\gamma_n(\lambda')\} = \gamma_{m-1}(\lambda)\,
\gamma_{n-1}(\lambda')\, \gamma_1( [\lambda,\lambda'])$ for all
$\lambda,\lambda'\in L$ and all $m,n\geq 1$.
\end{enumerate}
Moreover, $\{-,-\}$ is graded of degree $-1$.
\end{prop}

\begin{proof} We know (\cite[Chap. III, 1]{roby_63}) that $\Gamma_A L$ can be realized as the quotient of the polynomial algebra
$ R=A\left[ \{x_{\lambda,n}\}_{\lambda \in L, n\geq 0}\right] $ by
the ideal $I$ generated by the elements:
\begin{enumerate}
\item[(a)] $x_{\lambda,0}-1$,\quad $\lambda\in L$,
\item[(b)] $x_{a\lambda,m}-a^m x_{\lambda,m}$,\quad $\lambda\in L$, $a\in A$, $m\geq 0$,
\item[(c)] $x_{\lambda,m} x_{\lambda,n} - \binom{m+n}{m} x_{\lambda,m+n}$,\quad  $\lambda\in L$, $m,n\geq 0$,
\item[(d)] $\displaystyle x_{\lambda+\lambda',m} - \sum_{i+j=m} x_{\lambda,i} x_{\lambda',j}$,\quad  $\lambda,\lambda'\in L$, $m\geq 0$,
\end{enumerate}
and the maps $\gamma_n:L \to \Gamma_A L$ are given by
$\gamma_n(\lambda) = x_{\lambda,n} + I$. We consider $R$ as a graded
$A$-algebra, with $\deg(A) = 0$ and $\deg(x_{\lambda,m})=m$. The
ideal $I$ is clearly homogeneous and $\Gamma_A L$ is also a graded
$A$-algebra.
\medskip

We define a $k$-biderivation $\{-,-\}': R \times R \to R$ by:
\begin{enumerate}
\item[-)] $\{a,b\}'=0$ for all $a,b\in A$.
\item[-)] $\{a,x_{\lambda,m}\}' = -\{x_{\lambda,m},a\}' = - \lambda(a) x_{\lambda,m-1}$, for all $a\in A$, $\lambda\in L$ and $m\geq 0$, where we write  $x_{\lambda,-1}=0$.
\item[-)] $\{x_{\lambda,m},x_{\mu,n}\}' = x_{\lambda,m-1}\,
x_{\mu,n-1}\, x_{[\lambda,\mu],1}$  for all $\lambda,\mu\in L$ and
all $m,n\geq 0$.
\end{enumerate}
One can check that $\{r,r\}' = 0$ for all $r\in R$, and so
$\{-,-\}'$ is skew-symmetric, and that the Jacobi identity holds:
$$
\{r,\{s,t\}'\}' + \{s,\{t,r\}'\}' + \{t,\{r,s\}'\}' = 0
$$
for all $r,s,t\in R$. So $\{-,-\}'$ defines a Poisson structure on
$R$, which is clearly graded of degree $-1$.
\medskip

One can also check that $\{r,r'\}' \in I$ whenever $r\in I$ or
$r'\in I$, and so $\{-,-\}'$ passes to the quotient and defines a
Poisson structure $\{-,-\}$ on $\Gamma_A L$ satisfying properties
(i), (ii) and (iii). It is also graded of degree $-1$.
\medskip

Since $\Gamma_A L$ is generated as $\Z$-algebra by $a\in A$ and
$x_{\lambda,n}$ for $\lambda \in L, n\geq 0$, the above properties
determine $\{-,-\}$.
\end{proof}

\begin{prop} The maps of graded  $A$-algebras $\vartheta^f$ and $\vartheta^\infty$ above are
maps of Poisson algebras.
\end{prop}

\begin{proof} It is enough to treat the case of $\vartheta^f$.
It is clear that $\vartheta^f(\{a,a'\}) = 0 = \{a,a'\} =
\{\vartheta^f(a),\vartheta^f(a')\}$ for all $a,a'\in A$. It remains
to prove that:
\begin{enumerate}
\item[(a)] $\vartheta^f(\{\gamma_m(\delta),a\})  = \{\vartheta^f(\gamma_m(\delta)),a\}$ for all $\delta\in \IDer^f_k(A)$, all $a\in A$ and all $m\geq 1$.
\item[(b)] $\vartheta^f(\{\gamma_m(\delta),\gamma_n(\delta')\})  = \{\vartheta^f(\gamma_m(\delta)),\vartheta^f(\gamma_n(\delta'))\}$ for all
$\delta,\delta'\in \IDer^f_k(A)$ and all $m,n\geq 1$.
\end{enumerate}
For (a), let us take an $m$-integral $D\in\HS_k(A;m)$ of $\delta$. We
have:
\begin{align*}
& \vartheta^f(\{\gamma_m(\delta),a\}) = \vartheta^f(\delta(a)\,
\gamma_{m-1}(\delta))= \delta(a)\,\vartheta^f( \gamma_{m-1}(\delta))
= \delta(a)\, \sigma_{m-1}(D_{m-1}) = \sigma_{m-1}(D_1(a) D_{m-1}) =
&
\\
& \sigma_{m-1}([D_m,a]) = \{\sigma_m(D_m),a\} =
\{\vartheta^f(\gamma_m(\delta)),\vartheta^f(a)\}. &
\end{align*}
For (b), let us take an $m$-integral $D\in\HS_k(A;m)$ of $\delta$ and
an $n$-integral $D'\in\HS_k(A;m)$ of $\delta'$. We have:
\begin{align*}
& \{\vartheta^f(\gamma_m(\delta)),\vartheta^f(\gamma_n(\delta'))\} =
\{\sigma_m(D_m), \sigma_n(D'_n)\} = \sigma_{m+n-1}([D_m,D'_n]), &
\\
& \vartheta^f(\{\gamma_m(\delta),\gamma_n(\delta')\}) =
\vartheta^f(\gamma_{m-1}(\delta)\, \gamma_{n-1}(\delta')\,
\gamma_1([\delta,\delta'])) = \vartheta^f(\gamma_{m-1}(\delta))\,
\vartheta^f(\gamma_{n-1}(\delta'))\, \vartheta^f(
\gamma_1([\delta,\delta']))= &
\\
& \sigma_{m-1}(D_{m-1})\, \sigma_{n-1}(D'_{n-1})\,
\sigma_1([D_1,D'_1]) = \sigma_{m+n-1}(D_{m-1}\, D'_{n-1}\,
[D_1,D'_1]), &
\end{align*}
and the result is a consequence of Lemma \ref{lem:aux}.
\end{proof}

\begin{lem} \label{lem:aux}
For any HS-derivations $D\in\HS_k(A;m)$, $D'\in\HS_k(A;n)$, with
$m,n\geq 1$, the differential operator
$$ [D_m,D'_n] - D_{m-1} D'_{n-1} [D_1,D'_1]
$$
has order $\leq m+n-2$.
\end{lem}

\begin{proof}
We proceed by induction on $m+n$. Details are left to the reader.
\end{proof}

\bigskip

\noindent {\small \href{http://personal.us.es/narvaez/}{Luis Narv\'aez Macarro}, Mar\'{\i}a de la Paz Tirado Hern\'andez\\
\noindent \href{http://departamento.us.es/da/}{Departamento de
\'Algebra} \&\
\href{http://www.imus.us.es}{Instituto de Matem\'aticas (IMUS)}\\
\href{http://matematicas.us.es}{Facultad de Matem\'aticas}, \href{http://www.us.es}{Universidad de Sevilla}\\
Calle Tarfia s/n, 41012  Sevilla, Spain} \\
{\small {\it E-mail}\ : narvaez@us.es, mtirado@us.es}

\end{document}